\documentclass{paper}
\usepackage{amsmath,amsthm,mleftright}
\usepackage{amssymb,mathtools}
\DeclareMathAlphabet\mathcal{OMS}{cmsy}{m}{n}
\DeclareMathAlphabet\mathbfcal{OMS}{cmsy}{b}{n}
\usepackage{tikz-cd}
\usepackage{xspace}
\usepackage{makecell}
\usepackage{graphicx,booktabs}
\usepackage{ragged2e}
\usepackage{booktabs,multirow,tabularx}
\usepackage{paralist}
\usepackage{hyperref}
\hypersetup{%
    pdfmenubar=true,       % show Acrobat's menu?
    pdffitwindow=false,     % window fit to page when opened
    pdfstartview={FitH},    % fits the width of the page to the window
    pdftitle={Image Reconstruction for Multispectral CT},
    colorlinks=true,       % false: boxed links; true: colored links
    linkcolor=red,          % color of internal links
    citecolor=red,        % color of links to bibliography
    filecolor=magenta,      % color of file links
    urlcolor=cyan,           % color of external links
    pdfborder = {0,0,0}
}
\usepackage{acro}
\usepackage{cleveref}
\usepackage{algorithm}
\usepackage[noend]{algpseudocode}
\makeatletter
\def\BState{\State\hskip-\ALG@thistlm}
\makeatother
\Crefname{equation}{}{}
\crefname{equation}{}{}
\usepackage[title]{appendix}
\usepackage{todonotes}
\usepackage{inputenc}
\theoremstyle{plain}
 \newtheorem{theorem}{Theorem}
 \newtheorem{proposition}{Proposition} 
 \newtheorem{lemma}{Lemma} 
 \newtheorem{corollary}{Corollary}

\theoremstyle{definition}

\theoremstyle{remark}
  \newtheorem{remark}{Remark}

\DeclareMathOperator*{\diag}{\texttt{diag}}

\newcommand{\bdomega}{\boldsymbol{\omega}}

\newcommand{\bphi}{\boldsymbol{\phi}}
\newcommand{\bx}{\boldsymbol{x}}
\newcommand{\by}{\boldsymbol{y}}

\newcommand{\bdf}{\boldsymbol{f}}
\newcommand{\bdg}{\boldsymbol{g}}
\newcommand{\bdh}{\boldsymbol{h}}

\newcommand{\bdu}{\boldsymbol{u}}
\newcommand{\bdz}{\boldsymbol{z}}
\newcommand{\bLambda}{\boldsymbol{\Lambda}}

\newcommand{\bzero}{\boldsymbol{0}}

\newcommand{\rmd}{\mathrm{d}}
\newcommand{\Diff}{\mathcal{D}} % Differentiation
\newcommand{\tra}{\boldsymbol{\mathsf{T}}}
\newcommand{\Real}{\mathbb{R}}
\newcommand{\bdK}{\boldsymbol{K}}

\newcommand{\Xray}{\mathcal{P}}
\newcommand{\bXray}{\boldsymbol{\mathcal{P}}}
\newcommand{\bvarphi}{\boldsymbol{\varphi}}
\newcommand{\bdr}{\boldsymbol{r}}

\newcommand{\bddelta}{\boldsymbol{\delta}}

\newcommand{\domain}{\Omega}

\newcommand{\mtbf}{\mathbf{f}}
\newcommand{\mtbg}{\mathbf{g}}
\newcommand{\mtbK}{\mathbf{K}}
\newcommand{\mtbp}{\mathbf{p}}
\newcommand{\mtbb}{\mathbf{b}}
\newcommand{\mtbs}{\mathbf{s}}
\newcommand{\mtbz}{\mathbf{z}}
\newcommand{\mtbw}{\mathbf{w}}
\newcommand{\mtbP}{\mathbf{P}}
\newcommand{\mtbh}{\mathbf{h}}

\begin{document}

\title{AFIRE: Accurate and Fast Image Reconstruction Algorithm for Geometric-inconsistent Multispectral CT}
%\subtitle{Version 1.8, Nov. 19, 2024}
\author{Yu Gao\thanks{State Key Laboratory of Mathematical Sciences, Academy of Mathematics and Systems Science, Chinese Academy of Sciences, Beijing 100190, China; University of Chinese Academy of Sciences, Beijing 100190, China.} and Chong Chen\footnotemark[1]}

\maketitle
%\tableofcontents

\begin{abstract}
For nonlinear multispectral computed tomography (CT), accurate and fast image reconstruction is challenging when the scanning geometries under different X-ray energy spectra are inconsistent or mismatched. Motivated by this, we propose an Accurate and Fast Image REconstruction (AFIRE) algorithm to address such problems in the case of mildly full scan. From the continuous (resp. discrete) setting, we discover that the derivative operator (gradient) of the involved nonlinear mapping at some special points, for example, at zero, can be represented as a composition (block multiplication) of a diagonal operator (matrix) composed of X-ray transforms (projection matrices) and a very small-scale matrix. Based on these insights, the AFIRE algorithm is proposed by leveraging the simplified Newton method. Under proper conditions, we establish the convergence theory of the proposed algorithm. Furthermore, numerical experiments are also carried out to verify that the proposed algorithm can accurately and effectively reconstruct the basis images in completely geometric-inconsistent dual-energy CT with noiseless and noisy projection data. Particularly, the proposed algorithm significantly outperforms some state-of-the-art methods in terms of accuracy and efficiency. Finally, the flexibility and extensibility of the proposed algorithm are also demonstrated.    
\end{abstract}

% REQUIRED

\section{Introduction}\label{sec:Introduction}

Multispectral CT (MSCT) is an imaging modality aiming to reconstruct material-specific (basis) images from multiple sets of attenuation data collected with different X-ray energy spectra. These basis images can subsequently be used to form virtual monochromatic images (VMIs) or the spatial distributions of linear attenuation coefficient (LAC) at energies of interest by linear combination \cite{Hounsfield73CT,alma76,DE_MECT15}. MSCT has attracted increasing attention and research interest due to the development of hardware and the requirement of applications, as well as its advantages in considering the X-ray energy-dependent information over the conventional CT such as artifact suppression, contrast improvement and object separation \cite{DE_MECT15}. Dual-energy CT (DECT), as the first generation of MSCT, has been implemented in clinical practice for about twenty years \cite{SPCT_23review}. Despite the long history of DECT, its clinical and technical research is still in its infancy \cite{hsieh15}.

It is well-known that the forward data model used in conventional CT assumes that all of the X-ray photons emitting from an X-ray source are monochromatic, namely, have the same energy \cite{Shepp1978CT}. This assumption makes the model linear, which causes the image reconstruction to become a linear inverse problem, but leading to difficulties in artifact suppression and limitations in material differentiation. In real-world CT, however, this assumption is not realistic because the output X-ray beams are polychromatic, consisting of photons with a wide range of energies. The distribution of photon energies is referred to as the energy spectrum. The LAC of the scanned material depends not only on its own elemental composition but also on the energy of the probing X-ray photons. Hence, the forward data model of MSCT may be described as a nonlinear mapping, and the corresponding image reconstruction can be formulated as a nonlinear inverse problem \cite{alma76,maal76,mathCT01,hsieh15,DE_MECT15,Scherzer08}. 

Although development of algorithms for reconstructing basis images in MSCT constitutes an important topic of active research, accurate and fast image reconstruction remains a challenging task, largely because its forward data model is nonlinear. In this work, we understand different reconstruction algorithms from the perspective of scanning geometric parameters under different X-ray energy spectra. 

One situation is that the scanning geometric parameters are consistent or matched under different energy spectra. That is to say, for each X-ray path, under every energy spectrum, there is always one X-ray beam passing through it. Here this case is called geometric-consistent MSCT, in which the image reconstruction can be accurately realized by a two-step method: step (i) estimates the basis sinograms from measurements through solving a lot of very small-scale independent systems of nonlinear equations along different X-ray paths, for instance, using Newton's method, whereas step (ii) reconstructs basis images from the basis sinograms estimated, such as using the filtered back-projection (FBP) method. This method is commonly referred to as two-step data-domain-decomposition (DDD) method. As step (ii) involves the inversion of a linear system as in conventional CT, step (i) constitutes the key step of the DDD method in which a nonlinear data model needs to be inverted. Although many algorithms can invert the nonlinear model accurately, analysis of its solution existence, uniqueness and stability remains difficult. In recent years, several works studied the solution properties such as \cite{Lev_nonunique_07,Alvarez19Invert,Bal_2020,Ding21_invert,gao23solution}. 

The other is that the scanning geometric parameters are inconsistent or mismatched under different energy spectra. Namely, there is at least one X-ray path that all the X-ray beams under some energy spectrum do not pass through it. We call this case the geometric-inconsistent MSCT. A more extreme example is that for every X-ray path, there is always an energy spectrum under which no X-ray beams  pass through it, which can be seen as completely geometric-inconsistent MSCT. The configurations of many commercial devices fall into this situation, such as the fast kV-switching DECT of GE \cite{DE_MECT15}. 

The development of accurate and fast reconstruction algorithms for this case is more challenging than the geometric-consistent case. The authors in \cite{Zou08} matched the originally mismatched neighboring projections of high- and low-energy spectra in DECT, and then proposed a two-step DDD method to invert those one-view-deviated projection pairs. Similarly, using the matched projections as in \cite{Zou08}, an iterative FBP (IFBP) reconstruction method was proposed in \cite{zhao_19}, which applies a single Newton iteration and FBP alternately. However, due to the different geometric parameters of the matched projections, these approaches cannot result in an  accurate image reconstruction. 

In the last two decades, progress has been made in developing one-step methods for the image reconstruction in MSCT. To reconstruct the basis images directly, such kind of method often needs to devise an optimization model from a theoretical perspective, such as Bayesian inference and variational regularization, and then develop an iterative scheme to solve that model. For instance, the authors in \cite{elbakri2002statistical, cai2013full, zhang2013model, LongFessler14, mechlem2017joint} employed the maximum a posteriori estimation within  Bayesian inference to formulate the optimization models, and the authors in \cite{barber2016algorithm,chpan17,chpan21,gao22EPD} constructed the optimization models from the perspective of variational regularization. Because of the nonlinear data model in MSCT, the optimization models designed are generally nonconvex. While these methods can often achieve high image quality through a large number of iterations, they are time-consuming,  and their convergence analyses were mostly not considered. In addition, many studies also consider the problem as solving a large-scale system of nonlinear equations \cite{zhaozz14,gao23NKM}.

In this article, we aim to develop an image reconstruction algorithm of the one-step type that takes into account both accuracy and efficiency, has mathematical guarantee, and is suitable for the geometric-inconsistent MSCT.

The rest of this paper is organized as follows. In \cref{sec:Data-models}, we introduce the forward data models of MSCT. In \cref{sec:propose_alg}, we propose a novel algorithm to solve the image reconstruction problem in geometric-inconsistent MSCT. We then prove the convergence of the proposed algorithm in \cref{sec:convergence_analysis}. In \cref{sec:Numerical}, numerical experiments are carried out to illustrate the performance of the proposed algorithm. The flexibility and extensibility of the proposed algorithm are discussed in \cref{sec:discussion}. Finally, we conclude the paper in \cref{sec:conclucsion}.

\section{Forward data models in multispectral CT}
\label{sec:Data-models} 

In this section, the forward data models in MSCT are described briefly, including the continuous-to-continuous (CC)-data model and the discrete-to-discrete (DD)-data model \cite{alma76,hsieh15,chpan17,gao23solution}.

\subsection{CC-data model}
\label{sec:cc-data-model}
The data in MSCT are acquired from an object scanned with multiple X-ray energy spectra $\{s^{[q]}\}_{q=1}^Q$, where $s^{[q]}$ is the normalized nonnegative distributed function of the $q$-th energy spectrum, and $Q$ denotes the number of used distinct spectra. 

For some given spectrum{\footnote{The effective spectrum of a CT system is indeed the product of the X-ray-tube spectrum and detector-energy response.}}, data can be detected along the X-ray specified by $L(\theta^{[q]})$, where $\theta^{[q]}$ determines the ray direction. The detected data in the absence of other physical factors can be modeled as 
\begin{equation}\label{eq:multimeasurement}
g_{L(\theta^{[q]})} = \ln\int_{0}^{E_{\text{max}}} s^{[q]}(E)\exp\left(-\int_{L(\theta^{[q]})} \mu(\by, E) \rmd l\right) \rmd E,
\end{equation}
where $s^{[q]}$ is nonnegative and satisfying 
\[
\int_{0}^{E_{\text{max}}}s^{[q]}(E) \rmd E = 1,
\] 
and $\mu(\by, E)$ represents the energy-dependent LAC of interest at spatial position $\by \in \Real^n$ ($n=2$ or $3$) and energy $E\in [0, E_{\text{max}}]$, and $\Real^n$ the real Euclidean space, $E_{\text{max}}$ the maximum energy in the scan \cite{BuzugCT08,hsieh15,MSCTbook22}. The function $\mu$ of LAC is decomposed often into a linear combination of basis images $\{f_d\}_{d=1}^D$ as 
\begin{equation}\label{eq:attenuationco}
\mu(\by, E) = \sum_{d=1}^D b_d(E) f_d(\by), 
\end{equation}
where $b_d(E)$ denotes the expansion coefficient, and $D$ the number of basis images \cite{alma76,hsieh15}. For example, for the typical X-ray energy range in diagnostic DECT, $D=2$ is chosen often as photoelectric effect and Compton scatter are the dominant factors contributing to $\mu$ in the physical effect decomposition model. For the basis material decomposition model, $b_d$ denotes the known function of mass attenuation coefficient (MAC) of the $d$-th basis material (e.g., water and bone); and $f_d$ represents the material-equivalent density function of the $d$-th basis material. Importantly, these two models are physically equivalent \cite{hsieh15}. Consequently, using two basis materials is sufficient to accurately describe the linear attenuation characteristics. 

Substituting \cref{eq:attenuationco} into \cref{eq:multimeasurement}, we obtain 
\begin{equation}\label{eq:cc_formula}
g_{L(\theta^{[q]})}= \ln\int_{0}^{E_{\text{max}}} s^{[q]}(E)\exp\left(-\sum_{d=1}^D b_d(E) \int_{L(\theta^{[q]})}f_d(\by) \rmd l \right)\rmd E. 
\end{equation}
The formula in \cref{eq:cc_formula} is referred to as the CC-data model because $\theta^{[q]}$, $E$, and $\by$ are continuous variables \cite{chpan17}.

\subsection{DD-data model}
\label{sec:dd_model}

In practice, data can be detected only for discrete energies and rays, and an image is reconstructed only on an array of discrete pixels/voxels. Hence, it is possible to develop a DD-data model, which may be constructed by implementing a specific discrete scheme into the CC-data model in \cref{eq:cc_formula} \cite{chpan17}. However, in general, there is no direct connection between the two models. 

Assume that the energy domain $[0, E_{\text{max}}]$ is divided into $M$ intervals with equal size $\Delta_E$; the data $g_j^{[q]}$ are detected for discrete directions $\theta_j^{[q]}$ with $j = 1, \dots, J_q$ under the $q$-th spectrum; each basis image $f_d$ is discretized by an array of size $I$. Then a DD-data model is obtained as
\begin{equation}\label{eq:dd_formula}
g_j^{[q]}  = \ln\sum_{m=1}^{M} s^{[q]}_{m} \exp\left(-\sum_{d=1}^D b_{dm} \sum_{i=1}^I p^{[q]}_{ji} f_{di}
\right) \quad \text{for $j = 1, \dots, J_q$,}
\end{equation}
where $s^{[q]}_{m} = s^{[q]}(m\Delta_E)\Delta_E$ satisfying $\sum_{m=1}^{M} s^{[q]}_{m} =1$, $b_{dm} = b_d(m\Delta_E)\Delta_E$, $f_{di}$ denotes the gray value of the $d$-th discrete basis image at voxel $i$, and $p^{[q]}_{ji}$ is often taken as the intersection length of ray $j$ with pixel/voxel $i$ of the image array \cite{chen2020fast}. Here the DD-data model \cref{eq:dd_formula} can be seen as a specific discrete scheme for the CC-data model \cref{eq:cc_formula}.

\section{The proposed algorithm}\label{sec:propose_alg}

In this section, we introduce the MSCT image reconstruction problems and propose new algorithms for the continuous and discrete settings, which correspond to the CC-data model \cref{eq:cc_formula} and the DD-data model \cref{eq:dd_formula}, respectively. 

\subsection{The continuous setting}\label{subsec:continuous_MSCT}

Here we give the MSCT image reconstruction problem in continuous setting based on the CC-data model \cref{eq:cc_formula}, and then present the corresponding algorithm. 

%%state follows in appendix?
Before that, we introduce the required notation. Let $\mathcal{X}_1,\ldots,\mathcal{X}_D$ be normed vector spaces with norms $\|\cdot\|_{\mathcal{X}_1},\ldots,\|\cdot\|_{\mathcal{X}_D}$, respectively. Their Cartesian product space is given by 
\[
\mathbfcal{X}^D := \mathcal{X}_1\times\cdots\times\mathcal{X}_D = \{[u_1,\ldots,u_D]^{\tra} | u_d\in \mathcal{X}_d, ~ d=1,\ldots,D\}.
\] 
The norm of an element $\bdu=[u_1,\ldots,u_D]^{\tra}$ in this Cartesian product space can be defined as 
\[
\|\bdu\|_{\mathbfcal{X}^D}:=\left(\sum_{d=1}^D \|u_d\|^2_{\mathcal{X}_d}\right)^{1/2}.
\] 
Let $H^{\alpha}(\Real^n)$ be the Sobolev space of order $\alpha$ on $\Real^n$, and $\domain$ be a sufficiently regular open bounded subset in $\Real^n$. Denote  $H_0^{\alpha}(\domain) = \{u\in H^{\alpha}(\Real^n) | \text{supp}(u)\subset \domain\}$, where $\text{supp}(u) := \overline{\{\by\in \Real^n | u(\by) \neq 0\}}$ is the support of $u$.  

\subsubsection{Continuous MSCT image reconstruction problem}\label{subsubsec:cc_problem}

For simplicity, define $\mathbb{T}^n_q := \{(\bdomega, \bx) | \bdomega \in \mathbb{S}^{n-1}_q\subseteq \mathbb{S}^{n-1},\bx \in \Omega_q \subseteq \bdomega^{\bot}\}\subseteq \{(\bdomega, \bx) |\bdomega \in \mathbb{S}^{n-1},\bx \in \bdomega^{\bot}\}=:\mathbb{T}^n$ as the set of scanning geometric parameters under the $q$-th energy spectrum, where $q = 1, \ldots, Q$. Here $\mathbb{S}^{n-1}$ represents the unit sphere in $\Real^n$ and $\bdomega^{\bot}$ denotes the hyperplane perpendicular to $\bdomega$. 
  
Let 
\begin{equation}\label{eq:xray_transform}
    \Xray_q f(\bdomega, \bx) :=\int_{-\infty}^{+\infty} f(\bx + t\bdomega) ~\rmd t, \quad (\bdomega, \bx) \in \mathbb{T}^n_q, 
\end{equation}
which is the X-ray transform of a certain function $f$ defined on $\mathbb{R}^n$, equivalently the Radon transform if $n = 2$ \cite{mathCT01}. Note that if $\mathbb{T}^n_1 = \cdots = \mathbb{T}^n_Q$, then $\Xray_1 = \cdots = \Xray_Q$, this situation is called geometric consistency, otherwise it is called geometric inconsistency. 

For each energy spectrum $q$, denote the X-ray transform in \cref{eq:xray_transform} by the mapping 
\[
\Xray_q : \mathcal{F} = H_0^{-1/2}(\domain) \longrightarrow \tilde{\mathcal{G}}_q = \Xray_q (\mathcal{F}) \subseteq \mathcal{L}^2(\mathbb{T}^n_q) = \mathcal{G}_q. 
\]
As the X-ray $L(\theta^{[q]})$ in \cref{eq:cc_formula} can be parametrized by $(\bdomega, \bx)$, we translate \cref{eq:xray_transform} into  
\begin{equation}\label{eq:xray_transform_fd}
    \Xray_q f(\bdomega, \bx) = \int_{L(\theta^{[q]})} f(\by) ~\rmd l. 
\end{equation}
Using \cref{eq:xray_transform_fd}, the right-hand side of \cref{eq:cc_formula} can be rewritten as  
\begin{equation}\label{eq:imaging_operator}
    K^{[q]}(\bdf) := \ln \int_{0}^{E_{\max}} s^{[q]}(E)\exp\left(-\sum_{d=1}^D  b_d(E) \Xray_q f_{d} \right) \rmd E, 
\end{equation}
where $\bdf = [f_{1}, \ldots, f_{D}]^{\tra}$. 

Consequently, the aim of the continuous MSCT image reconstruction is to reconstruct $\bdf$ from the measured data $\bdg = [g^{[1]},\ldots,g^{[Q]}]^{\tra}$ such that
\begin{equation}\label{eq:measurements}
K^{[q]}(\bdf) = g^{[q]}\quad \text{for}~ q=1,\ldots, Q. 
\end{equation}
Note that \cref{eq:measurements} can be expressed as the following compact form 
\begin{equation}\label{eq:cc_imaging_prob}
   \bdK(\bdf) = \bdg, 
\end{equation}
where 
$\bdK(\bdf) = [K^{[1]}(\bdf),\ldots,K^{[Q]}(\bdf)]^{\tra}$.

\subsubsection{The proposed algorithm for the continuous problem}\label{subsubsec:cc_alg}

Here we propose the AFIRE algorithm for solving the geometric-inconsistent MSCT image reconstruction problem in the continuous case. Let $\mathbfcal{F}^D = \mathcal{F}\times\cdots\times\mathcal{F}$, $\mathbfcal{G}^Q = \mathcal{G}_1\times\cdots\times\mathcal{G}_Q$, and $\tilde{\mathbfcal{G}}^Q = \tilde{\mathcal{G}}_1\times\cdots\times \tilde{\mathcal{G}}_Q$. 

First of all, we have the following proposition. 

\begin{proposition}\label{prop:Frechet_derivative}
Let $\{s^{[q]}\}_{q=1}^Q$ be nonnegative normalized functions and $\{b_d\}_{d=1}^D$ be nonnegative bounded functions. Then for each $q$, the functional in \cref{eq:imaging_operator} can be represented as a mapping $K^{[q]} : \mathbfcal{F}^D \rightarrow \mathcal{G}_q$ and is G\^ateaux differentiable, and its G\^ateaux derivative at $\bdf \in \mathbfcal{F}^D$ for $\bdh = [h_1,\ldots,h_D]^{\tra} \in \mathbfcal{F}^D$ is computed as 
\begin{equation}\label{eq:Frech_deriva}
\Diff K^{[q]}(\bdf, \bdh) = - \int_0^{E_{\max}}  \Phi^{[q]}(\bdf; E) \left( \sum_{d=1}^D b_d(E) \Xray_q h_d\right) \rmd E, 
\end{equation}
where 
\begin{equation}\label{eq:Phi_f}
\Phi^{[q]}(\bdf; E) = \frac{\displaystyle s^{[q]}(E)\exp\left(-\sum\limits_{d=1}^D  b_d(E) \Xray_q f_d\right)}{\displaystyle\int_0^{E_{\max}}  s^{[q]}(E) \exp\left(-\sum\limits_{d=1}^D  b_d(E) \Xray_q f_d\right) \rmd E}.
\end{equation}
Moreover, both the $K^{[q]}$ and $\bdK : \mathbfcal{F}^D \rightarrow \mathbfcal{G}^Q$ are Fréchet differentiable.  
\end{proposition}
\begin{proof}
    See \cref{proof:G_F_diff} for the proof.
\end{proof} 

\paragraph{Our insight} We start with analyzing the G\^ateaux derivative of $K^{[q]}$ in \cref{eq:Frech_deriva}. Notice that the $\Phi^{[q]} (\bdf; E)$ in \cref{eq:Phi_f} is actually related to the geometric parameter $(\bdomega,\bx)\in \mathbb{T}^n_q$ since the term $\Xray_q f_d$ depends on that parameter as in \cref{eq:xray_transform_fd}. For a certain real constant set $C_d^{[q]}\}_{d=1}^D$, if there exists one $\hat{\bdf}=[\hat{f}_1,\ldots,\hat{f}_D]^{\tra} \in \mathbfcal{F}^D$ such that  
\begin{equation}\label{eq:cc_special_points}
    \Xray_q \hat{f}_d(\bdomega,\bx) \equiv C_d^{[q]}, \quad\forall (\bdomega,\bx)\in \mathbb{T}^n_q, 
\end{equation}
then $\Phi^{[q]} (\hat{\bdf}; E)$ is independent of the geometric parameter $(\bdomega,\bx)$. Together with the linearity of $\Xray_q$, the G\^ateaux derivative at $\hat{\bdf}$ for $\bdh$ can be rewritten as
\begin{equation}\label{eq:linear_DK}
    \Diff K^{[q]} (\hat{\bdf},\bdh) = - \Xray_q\left(\sum_{d=1}^D \varphi^{[q]}_d h_d\right). 
\end{equation}
Here  
\begin{equation}\label{eq:phi_q_d}
 \varphi^{[q]}_d = \int_0^{E_{\max}}   \Phi^{[q]} (\hat{\bdf}; E) b_d(E) ~\rmd E, 
\end{equation}
where 
\[ 
\Phi^{[q]} (\hat{\bdf}; E) = \frac{\displaystyle s^{[q]}(E)\exp\left(-\sum\limits_{d=1}^D  b_d(E) C_d^{[q]}\right)}{\displaystyle\int_0^{E_{\max}}  s^{[q]}(E) \exp\left(-\sum\limits_{d=1}^D  b_d(E) C_d^{[q]}\right) \rmd E}.
\]
\begin{remark}\label{rmk:special_point_CC}
    The simplest example is taking $\hat{\bdf} = \bzero\in\mathbfcal{F}^D$ for all of the $C_d^{[q]} = 0$, $d = 1, \ldots, D$ and $q = 1, \ldots, Q$ that satisfies the condition of \cref{eq:cc_special_points}. Although there may be other points that satisfy this condition, their explicit identification is not critical and is beyond the scope of the current research. This is because both the practical implementation and the theoretical convergence of the proposed algorithm are independent of the explicit form of such a point. 
\end{remark}

Following the insight above, the G\^ateaux derivative of $\bdK$ at the $\hat{\bdf}$ determined by \cref{eq:cc_special_points} can be given as the mapping below
\begin{align}\label{eq:map_0}
    \nonumber \Diff \bdK(\hat{\bdf},\cdot) : \mathbfcal{F}^D &\longrightarrow \tilde{\mathbfcal{G}}^Q \\
    \bdh &\longmapsto -\bXray(\bvarphi\bdh),
\end{align}
where $\bXray = \diag(\Xray_1, \ldots, \Xray_Q)$ denotes a diagonal operator composed of $Q$ X-ray transforms, and $\bvarphi = \bigl(\varphi^{[q]}_d\bigr)_{Q\times D}$ represents the constant matrix with entries $\varphi^{[q]}_d$ given in \cref{eq:phi_q_d}.  

To proceed, we study the invertibility of $\Diff \bdK(\hat{\bdf},\cdot)$. We have the following result. 

\begin{proposition}\label{prop:0_deriva_invert}
Assume that $Q = D$, the conditions in \cref{prop:Frechet_derivative} hold, and the X-ray transform $\Xray_q : \mathcal{F} \rightarrow \tilde{\mathcal{G}}_q$ for each $q$ and the matrix $\bvarphi$ are invertible. Then the $\Diff \bdK(\hat{\bdf},\cdot)$ in \cref{eq:map_0} is invertible. 
\end{proposition}
\begin{proof}
Since $\Xray_1, \ldots, \Xray_Q$ are invertible, $\bXray$ is invertible. Combining with the invertibility of $\bvarphi$, we obtain the invertibility of $\Diff \bdK(\hat{\bdf},\cdot)$.
\end{proof} 

\begin{remark}\label{rem:Q_D_equality}
If $\mathbb{S}^{n-1}_q \subseteq \mathbb{S}^{n-1}$ meets every equatorial circle of $\mathbb{S}^{n-1}$, it follows from \cite[theorem 2.11]{natterer2001mathematical} that $\Xray_q : \mathcal{F} \rightarrow \tilde{\mathcal{G}}_q$ is invertible. The condition of $Q=D$ means that the number of used energy spectra should be equal to the number of basis images, which is commonly satisfied in clinical applications. For example, $Q = D = 2$ in DECT. 
\end{remark}

Moreover, denote $\Xray^{-1}_q$ as the inverse of $\Xray_q$, and naturally
\[
\bXray^{-1} = \diag(\Xray_1^{-1}, \ldots, \Xray_Q^{-1})
\] 
is the inverse of $\bXray$. From \cref{prop:0_deriva_invert}, we can obtain the inverse mapping of $\Diff \bdK(\hat{\bdf},\cdot)$ as    
\begin{align}\label{eq:inverse_map}
\nonumber (\Diff \bdK)^{-1}(\hat{\bdf},\cdot) : \tilde{\mathbfcal{G}}^Q &\longrightarrow \mathbfcal{F}^D \\
\bdr &\longmapsto  - \bvarphi^{-1}\bXray^{-1}(\bdr). 
\end{align}
Hence, by leveraging the simplified Newton method \cite{Deuflhard_11,Iter_NonEq_book}, the proposed iterative scheme for solving \cref{eq:cc_imaging_prob} can be formulated as 
\begin{equation}\label{eq:iterative_scheme}
\bdf^{k+1} = \bdf^{k} - \bvarphi^{-1}\bXray^{-1}\bigl(\bdg - \bdK(\bdf^{k})\bigr). 
\end{equation}
Note that \cref{eq:iterative_scheme} is well-defined if $\bdg - \bdK(\bdf^{k})\in\tilde{\mathbfcal{G}}^Q$ for $\bdf^{k}\in\mathbfcal{F}^D$, and the standard simplified Newton method considers the derivative at the initial point $\bdf^0$. Based on the iterative scheme in \cref{eq:iterative_scheme}, we present the algorithm AFIRE for solving the continuous MSCT image reconstruction problem in \cref{alg:solve_cc_prob}. 

\begin{algorithm}[htb]
\caption{The proposed algorithm for solving the continuous problem in \cref{eq:measurements} (or \cref{eq:cc_imaging_prob}).}
\label{alg:solve_cc_prob}
\begin{algorithmic}[1]
\State \emph{Initialize}: Given $Q = D$, the measured data $\{g^{[q]}\}_{q=1}^Q$, the distributions of energy spectra $\{s^{[q]}\}_{q=1}^Q$, the MAC functions $\{b_d\}_{d=1}^D$ and a certain real constant set $\{C_d\}_{d=1}^D$. Initialize the basis image $\bdf^{0}$. Let $k \gets 0$. 

\State Compute the constant matrix $\bvarphi = (\varphi^{[q]}_d)_{Q\times D}$ by \cref{eq:phi_q_d}
and the inverse $\bvarphi^{-1}$. 
\State \emph{Loop}:
\State \quad Compute the residual 
\begin{equation*}
    \begin{bmatrix}
    \varLambda_1^k \\
    \vdots \\
    \varLambda_Q^k
    \end{bmatrix} \longleftarrow \begin{bmatrix}
    g^{[1]} - K^{[1]}(\bdf^{k}) \\
    \vdots \\
    g^{[Q]} - K^{[Q]}(\bdf^{k})
    \end{bmatrix}.
\end{equation*}

\State \quad Update $\bdf^{k+1}$ by
 \begin{equation*}
        \begin{bmatrix}
            f_1^{k+1} \\
            \vdots \\
            f_D^{k+1}
        \end{bmatrix} \longleftarrow \begin{bmatrix}
            f_1^{k} \\
            \vdots \\
            f_D^{k}
        \end{bmatrix}
            - \bvarphi^{-1}  
            \begin{bmatrix}
            \Xray^{-1}_1\varLambda_1^k \\
            \vdots \\
            \Xray^{-1}_Q\varLambda_Q^k
            \end{bmatrix}.
    \end{equation*}
\State \quad If some given termination condition is satisfied, then \textbf{output} $\bdf^{k+1}$; 
\newline \indent \hspace{-3.5mm}otherwise, let $k \gets k+1$, \textbf{goto} \emph{Loop}.
\end{algorithmic}
\end{algorithm}

\begin{remark}\label{rem:special_point}
    In \cref{alg:solve_cc_prob}, one can specifically choose the $C_d^{[q]} = 0$ for $d = 1, \ldots, D$ and $q = 1, \ldots, Q$ with $\hat{\bdf} = \bzero\in\mathbfcal{F}^D$ to fulfill the condition of \cref{eq:cc_special_points}. The corresponding formula of \cref{eq:phi_q_d} becomes 
    \[
    \varphi^{[q]}_d = \int_0^{E_{\max}}   s^{[q]}(E) b_d(E)~\rmd E. 
    \]
\end{remark}

\subsection{The discrete setting}\label{sec:discrete_MSCT}

Here, we start from the DD-data model \cref{eq:dd_formula} and then present the associated algorithm for solving the geometric-inconsistent MSCT image reconstruction problem under the mildly full-scan case.

\subsubsection{Discrete MSCT image reconstruction problem}\label{subsubsec:dd_prob}

As described in \cref{sec:dd_model}, for the discrete case, the forward operator of ray $j$ under the $q$-th energy spectrum may be formulated as 
\begin{equation}\label{eq:discrete_imaging_op}
K_j^{[q]}(\mtbf) = \ln\sum_{m=1}^{M} s^{[q]}_{m} \exp\left(-\sum_{d=1}^D b_{dm} \sum_{i=1}^I p^{[q]}_{ji} f_{di}\right).
\end{equation}
Here $\mtbf := [\mtbf_{1}^{\tra}, \ldots, \mtbf_{D}^{\tra}]^{\tra}$, where $\mtbf_{d}=[f_{d1},\ldots, f_{dI}]^{\tra}$ denotes the $d$-th discretized basis image. Hence, the discrete MSCT image reconstruction problem is to determine $\mtbf$ from the measured data $\{g_j^{[q]}\}_{j,q=1}^{J_q, Q}$ by the system of nonlinear equations as below
\begin{equation}\label{eq:discrete_imaging_prob}
    K_j^{[q]}(\mtbf) = g_j^{[q]} \quad\text{for}~j=1,\ldots,J_q~\text{and}~q=1, \ldots,Q. 
\end{equation}
Here $J_q$ is the number of rays under the $q$-th energy spectrum. The compact form of \cref{eq:discrete_imaging_prob} can be expressed as 
\begin{equation}\label{eq:dd_imaging_prob}
    \mtbK(\mtbf) = \mtbg. 
\end{equation}
Here 
\[
    \mtbK(\mtbf) = \begin{bmatrix}
        \mtbK^{[1]}(\mtbf) \\
        \vdots \\
       \mtbK^{[Q]}(\mtbf) 
        \end{bmatrix}, 
    \mtbK^{[q]}(\mtbf) = \begin{bmatrix}
        K_1^{[q]}(\mtbf) \\
        \vdots \\
        K_{J_q}^{[q]}(\mtbf) 
        \end{bmatrix}, 
    \mtbg = \begin{bmatrix}
        \mtbg^{[1]} \\
        \vdots \\
        \mtbg^{[Q]} 
        \end{bmatrix},
    \mtbg^{[q]} = \begin{bmatrix}
        g^{[q]}_1 \\
        \vdots \\
        g^{[q]}_{J_q} 
        \end{bmatrix}.     
\]

Note that if $J_1 = \cdots = J_Q$ and $p^{[1]}_{ji} = \cdots = p^{[Q]}_{ji}$ for any $(j, i)$, this situation is called geometric consistency, otherwise it is called geometric inconsistency. 

\begin{remark}
    In the geometric-consistent case, the same scanning geometries across different energy spectra allow \cref{eq:dd_imaging_prob} to be divided into a lot of very small-scale independent subsystems along each X-ray path, enabling high solving efficiency via the two-step DDD method \cite{gao23solution}. In contrast, the geometric inconsistency prevents this kind of division, requiring to directly solve a large-scale nonlinear system with $\sum_{q=1}^Q J_q$ equations (e.g., 294,912 equations considered in \cref{subsec:noiseless_forbild}). If the special structure is ignored, solving such a large-scale nonlinear problem with high precision is quite time-consuming. 
\end{remark}

\subsubsection{The proposed algorithm for the discrete problem}\label{subsubsec:dd_alg}

For brevity, let us introduce the following notation, for $q=1,\ldots,Q$ and $d=1,\ldots, D$, 
\begin{align*}
\mtbp_j^{[q]} &= [p_{j1}^{[q]}, \ldots, p_{jI}^{[q]}]^{\tra}, \quad \mtbb_d = [b_{d1}, \ldots, b_{dM}]^{\tra}, \quad \mtbs^{[q]} = [s_1^{[q]}, \ldots, s_M^{[q]}]^{\tra}, \\
z_{jm}^{[q]}(\mtbf) &= -\sum_{d=1}^{D}b_{dm}(\mtbp_j^{[q]})^{\tra}\mtbf_{d}, \quad \exp\bigl(\mtbz_j^{[q]}(\mtbf)\bigr) = \bigl[\exp\bigl(z_{j1}^{[q]}(\mtbf)\bigr), \ldots, \exp\bigl(z_{jM}^{[q]}(\mtbf)\bigr)\bigr]^{\tra}. 
\end{align*}
By \cref{eq:discrete_imaging_op}, the gradient of $K_j^{[q]}(\mtbf)$ is computed as 
\begin{equation}\label{eq:gradient_K_discr}
\nabla K_j^{[q]}(\mtbf) =  - \begin{bmatrix}
        \bigl(\mtbw_j^{[q]}(\mtbf)^{\tra}\mtbb_1\bigr)\mtbp_j^{[q]} \\
        \vdots \\
        \bigl(\mtbw_j^{[q]}(\mtbf)^{\tra}\mtbb_D\bigr)\mtbp_j^{[q]} 
        \end{bmatrix}, 
\end{equation}
where 
\begin{equation}\label{eq:omega_discr}
\mtbw_j^{[q]}(\mtbf) = \frac{\mtbs^{[q]}\odot\exp\bigl(\mtbz_j^{[q]}(\mtbf)\bigr)}{(\mtbs^{[q]})^{\tra}\exp\bigl(\mtbz_j^{[q]}(\mtbf)\bigr)}, 
\end{equation}
and $\odot$ denotes the operation of the elementwise product. 

\paragraph{Our insight} Similar to the continuous setting, for a certain real constant set $\{C_d^{[q]}\}_{d,q=1}^{D,Q}$, if there exists $\hat{\mtbf} = [\hat{\mtbf}_1^{\tra}, \ldots, \hat{\mtbf}_D^{\tra}]^{\tra}$ such that  
\begin{equation}\label{eq:dd_special_points}
    \bigl(\mtbp_j^{[q]}\bigr)^{\tra} \hat{\mtbf}_d \equiv C_d^{[q]}, \quad \forall j \in \{1, \ldots, J_q\}, 
\end{equation}
then, using \cref{eq:dd_special_points}, we have 
\[
z_{jm}^{[q]}(\hat{\mtbf}) = - \sum_{d=1}^{D}b_{dm}C_d^{[q]}, 
\]
and hence 
\begin{equation}\label{eq:exp_z_discr}
\exp\bigl(\mtbz_j^{[q]}(\hat{\mtbf})\bigr) = \Bigl[\exp\Bigl(- \sum_{d=1}^{D}b_{d1}C_d^{[q]}\Bigr), \ldots, \exp\Bigl(-\sum_{d=1}^{D}b_{dM}C_d^{[q]}\Bigr)\Bigr]^{\tra}. 
\end{equation}
By \cref{eq:exp_z_discr}, from \cref{eq:omega_discr}, it is easy to figure out that 
\[
\mtbw_j^{[q]}(\hat{\mtbf}) = \frac{\displaystyle\mtbs^{[q]}\odot\exp\bigl(- \sum_{d=1}^{D}b_{dm}C_d^{[q]}\bigr)}{\displaystyle(\mtbs^{[q]})^{\tra}\exp\bigl(- \sum_{d=1}^{D}b_{dm}C_d^{[q]}\bigr)}, 
\]
which is independent of the index $j$ (namely, the concrete ray). For short, let  
\begin{align*}
\hat{\mtbw}^{[q]} := \mtbw_j^{[q]}(\hat{\mtbf}).
\end{align*}
\begin{remark}
    The simplest example that satisfying the condition of \cref{eq:dd_special_points} is taking $\hat{\mtbf} = \bzero$, which leads to $C_d^{[q]} = 0$ for any $q$ and $d$. Similarly to \cref{rmk:special_point_CC}, the explicit identification of other such points is not critical and is beyond the scope of this work. 
\end{remark}

Considering \cref{eq:gradient_K_discr} at the $\hat{\mtbf}$ in \cref{eq:dd_special_points}, we obtain that 
\[ 
\nabla K_j^{[q]}(\hat{\mtbf}) =  - \begin{bmatrix}
        \phi^{[q]}_1\mtbp_j^{[q]} \\
        \vdots \\
       \phi^{[q]}_D\mtbp_j^{[q]} 
    \end{bmatrix}, 
    \] 
    where 
    \begin{equation}\label{eq:dd_phi}
    \phi^{[q]}_d := (\hat{\mtbw}^{[q]})^{\tra}\mtbb_d  
    %\bdw_j^{[q]}(\hat{\bdf})^{\tra}\bdb_d
    = \frac{\displaystyle\sum_{m=1}^{M}b_{dm}s_m^{[q]}\exp\bigl(- \sum_{d=1}^{D}b_{dm}C_d^{[q]}\bigr)}{\displaystyle\sum_{m=1}^{M}s_m^{[q]}\exp\bigl(- \sum_{d=1}^{D}b_{dm}C_d^{[q]}\bigr)}.
    \end{equation}
Then the Jacobian matrix of $\mtbK^{[q]}$ at $\hat{\mtbf}$ can be written as the following form 
\begin{align*} 
    \nabla \mtbK^{[q]}(\hat{\mtbf}) = - \begin{bmatrix}
    \phi^{[q]}_1\mtbP_q & \cdots& \phi^{[q]}_D\mtbP_q  
    \end{bmatrix}, 
\end{align*}
where $\mtbP_q = [\mtbp_1^{[q]}, \dots, \mtbp_{J_q}^{[q]}]^{\tra}$ is the projection matrix under the $q$-th energy spectrum. Note that $\mtbP_1 = \cdots = \mtbP_Q$ in the geometric-consistent case, but it fails to hold in the geometric-inconsistent case.

Furthermore, the associated Jacobian matrix of $\mtbK$ at $\hat{\mtbf}$ is 
\begin{equation}\label{eq:dd_nabla_K0}
    \nabla\mtbK(\hat{\mtbf}) =  -\begin{bmatrix}
    \phi^{[1]}_1\mtbP_{1} & \cdots & \phi^{[1]}_D\mtbP_{1} \\
    \vdots &\ddots  & \vdots \\
    \phi^{[Q]}_1\mtbP_{Q} & \cdots & \phi^{[Q]}_D\mtbP_{Q}
    \end{bmatrix} =  - \mtbP \otimes \bphi, 
\end{equation}
where $\mtbP = \diag(\mtbP_{1}, \ldots, \mtbP_{Q})$ denotes a diagonal matrix composed of $Q$ projection matrices, $\bphi = \bigl(\phi^{[q]}_d\bigr)_{Q\times D}$ represents the constant matrix with entries $\phi^{[q]}_d$, and the operation $\otimes$ represents the block multiplication between two matrices. 

As the previous continuous situation, we can also study the invertibility of $\nabla\mtbK(\hat{\mtbf})$. 
\begin{proposition}\label{prop:dd_0_deriva_invert}
Assume that $Q=D$, and $\bphi$, $\mtbP_q$ for $q=1, \ldots, Q$ are invertible. Then $\nabla\mtbK(\hat{\mtbf})$ is invertible.  
\end{proposition}
\begin{proof}
Since $Q=D$, using the invertibility of matrix, from \cref{eq:dd_nabla_K0}, it is easy to obtain that the inverse of $\nabla\mtbK(\hat{\mtbf})$ is  
\[
\bigl(\nabla\mtbK(\hat{\mtbf})\bigr)^{-1} = - \bphi^{-1} \otimes \mtbP^{-1}, 
\]
where $\mtbP^{-1} = \diag(\mtbP_{1}^{-1}, \ldots, \mtbP_{Q}^{-1})$. 
\end{proof} 

Again, by leveraging the simplified Newton method \cite{Deuflhard_11,Iter_NonEq_book}, the proposed iterative scheme for solving the discrete problem in \cref{eq:discrete_imaging_prob} (or \cref{eq:dd_imaging_prob}) can be formulated as 
\begin{equation}\label{eq:dd_iterative_scheme}
\mtbf^{k+1} = \mtbf^{k} - \bphi^{-1}\otimes \mtbP^{-1}\bigl(\mtbg - \mtbK(\mtbf^{k})\bigr). 
\end{equation}
Note that the standard simplified Newton method utilizes the Jacobian matrix at the initial point $\mtbf^{0}$. 

More clearly, based on \cref{eq:dd_iterative_scheme}, we present the proposed algorithm AFIRE for solving \cref{eq:discrete_imaging_prob} (or \cref{eq:dd_imaging_prob}) in \cref{alg:solve_dd_prob}, which may also be seen as a specific discretized counterpart of \cref{alg:solve_cc_prob}.
\begin{algorithm}[htbp]
    \caption{The proposed algorithm for solving the discrete problem in \cref{eq:discrete_imaging_prob} (or \cref{eq:dd_imaging_prob}).}
    \label{alg:solve_dd_prob}
    \begin{algorithmic}[1]
    \State \emph{Initialize}: Given $Q = D$, the measured data $\{\mtbg^{[q]}\}_{q=1}^Q$, the discrete energy spectra $\{\mtbs^{[q]}\}_{q=1}^Q$, the MAC $\{\mtbb_d\}_{d=1}^D$, and a certain real constant set $\{C_d^{[q]}\}_{d,q=1}^{D,Q}$. Initialized the basis images $\{\mtbf_{d}^0\}_{d=1}^D$. Let $k \gets 0$. 
    \State Compute the constant matrix $\bphi = \bigl(\phi^{[q]}_d\bigr)_{Q\times D}$ from \cref{eq:dd_phi} and the inverse $\bphi^{-1}$. 
    \State \emph{Loop}:
    \State \quad Compute the residual 
\begin{equation*}
    \begin{bmatrix}
            \bLambda_1^k \\
            \vdots \\
            \bLambda_Q^k
        \end{bmatrix}  \longleftarrow \begin{bmatrix}
            \mtbg^{[1]} - \mtbK^{[1]}(\mtbf^{k}) \\
            \vdots \\
            \mtbg^{[Q]} - \mtbK^{[Q]}(\mtbf^{k})
        \end{bmatrix}. 
\end{equation*}
    \State \quad Update $\mtbf^{k+1}$ by 
    \begin{equation}\label{eq:f_upd_dd}
        \begin{bmatrix}
            \mtbf_{1}^{k+1} \\
            \vdots \\
            \mtbf_{D}^{k+1}
        \end{bmatrix} \longleftarrow \begin{bmatrix}
            \mtbf_{1}^{k} \\
            \vdots \\
            \mtbf_{D}^{k}
        \end{bmatrix}
            - \bphi^{-1} \otimes 
            \begin{bmatrix}
            \mtbP_{1}^{-1}\bLambda_1^k \\
            \vdots \\
            \mtbP_{Q}^{-1}\bLambda_Q^k
            \end{bmatrix}. 
    \end{equation}
    \State \quad If some given termination condition is satisfied, then \textbf{output} $\mtbf^{k+1}$; \newline \indent \hspace{-3.5mm}otherwise, let $k \gets k+1$, \textbf{goto} \emph{Loop}.
    \end{algorithmic}
\end{algorithm}

 \begin{remark}\label{rem:dd_special_point}
     In \cref{alg:solve_dd_prob}, one can particularly choose the $\{C_d^{[q]}=0\}_{d,q=1}^{D,Q}$ with $\hat{\mtbf} = \bzero$ to fulfill the condition in \cref{eq:dd_special_points}. Then, the corresponding formula of \cref{eq:dd_phi} becomes 
    \[
    \phi^{[q]}_d = \sum_{m=1}^M b_{dm}s^{[q]}_m. 
    \] 
\end{remark}

 \begin{remark}\label{rem:dd_P_plus}
As a matter of fact, the inverse of the matrix $\mtbP$ in \cref{eq:dd_iterative_scheme} is difficult to calculate accurately. That is because all of its diagonal elements $\mtbP_{1},\ldots,\mtbP_{Q}$ may not even be invertible square matrices, even if they are, the computational cost of calculating the inverse of a large-scale matrix is very high. Therefore, we can only obtain an approximation for the inverse of $\mtbP$, and the actual iterative scheme used for \cref{eq:dd_iterative_scheme} should be   
\begin{equation}\label{eq:dd_scheme_approx}
\mtbf^{k+1} = \mtbf^{k} - \bphi^{-1}\otimes \mtbP^{+}\bigl(\mtbg - \mtbK(\mtbf^{k})\bigr), 
\end{equation}
where $\mtbP^{+} = \diag(\mtbP_{1}^{+}, \ldots, \mtbP_{Q}^{+})$ is the approximation for the inverse of $\mtbP$. Here $\mtbP_{q}^{+}$ denotes the approximation for the inverse of $\mtbP_{q}$. The update of \cref{eq:f_upd_dd} should be replaced with 
    \begin{equation}\label{eq:f_upd_dd_approx}
        \begin{bmatrix}
            \mtbf_{1}^{k+1} \\
            \vdots \\
            \mtbf_{D}^{k+1}
        \end{bmatrix} \longleftarrow \begin{bmatrix}
            \mtbf_{1}^{k} \\
            \vdots \\
            \mtbf_{D}^{k}
        \end{bmatrix}
            - \bphi^{-1} \otimes 
            \begin{bmatrix}
            \mtbP_{1}^{+}\bLambda_1^k \\
            \vdots \\
            \mtbP_{Q}^{+}\bLambda_Q^k
            \end{bmatrix}.
    \end{equation} 
\end{remark}

\begin{remark}\label{rem:inverse_comm} 
In practice, all the $\mtbP_{q}^{+}$, the approximation for the inverse of $\mtbP_{q}$, $q = 1, \dots, Q$, are independent of each other, which can be calculated in parallel by the iterative algorithms, such as algebraic reconstruction technique, conjugate gradient, and the other methods derived from optimization models. Moreover, the full-scan case in this study makes the method of FBP be applicable and efficient. 
\end{remark}

\section{Convergence analysis}\label{sec:convergence_analysis}

In this section, we study the convergence of the proposed algorithms that given in \cref{alg:solve_cc_prob,alg:solve_dd_prob}. Assuming that both the solutions exist to the continuous problem \cref{eq:cc_imaging_prob} and the discrete one \cref{eq:dd_imaging_prob}, and we refer to $\bdf^*$ and $\mtbf^*$ as the solutions, respectively.

Let $\|\cdot\|_F$ and $\sigma_{\min}(\cdot)$ be the Frobenius norm and the minimum singular value of a matrix, respectively. If $\sigma_{\min}(\cdot)$ is nonvanishing, define the scaled condition number by $\kappa_F(\cdot) := \|\cdot\|_F/\sigma_{\min}(\cdot)$. Denote $\|\cdot\|_1$ and $\|\cdot\|_2$ as the $1$-norm and Euclidean norm of a vector or the spectral norm of a matrix, respectively. 

\subsection{Convergence analysis for the continuous setting}

To proceed, we give the following lemma. 

\begin{lemma}\label{lem:continuous}
Suppose that the conditions in \cref{prop:0_deriva_invert} hold, and $\mathbb{T}^{n}_q = \mathbb{T}^{n}$ for $q=1,\ldots,Q$. Assume that there exists a constant $\eta > 0$ such that, for any nonnegative normalized function $\rho$ and any indices $q, d$, 
    \begin{equation*}
       \Bigl{|} \int_{0}^{E_{\max}} \left(\rho(E) - \Phi^{[q]}(\hat{\bdf}; E) \right) b_d(E)\rmd E \Bigr{|}\le \eta \Bigl{|}\int_{0}^{E_{\max}} \Phi^{[q]}(\hat{\bdf}; E) b_d(E)\rmd E\Bigr{|}.
    \end{equation*}
    Then, for any $\bdf, \bdh\in\mathbfcal{F}^D$, 
    \begin{equation}\label{eq:op_RGDC_0}
    \|\Diff \bdK(\bdf,\bdh)-\Diff \bdK(\hat{\bdf},\bdh)\|_{\mathbfcal{G}^Q}\le \eta \kappa_F(\bvarphi) \|\Diff \bdK(\hat{\bdf},\bdh) \|_{\mathbfcal{G}^Q}. 
    \end{equation} 
\end{lemma}

\begin{proof}
    See \cref{proof:cc_lemma} for the proof.
\end{proof}

Using the above result, we can establish the global convergence of \cref{alg:solve_cc_prob} and the uniqueness of the solution.
\begin{theorem}\label{thm:conver_cc_alg}
    Suppose that the conditions in \cref{lem:continuous} hold and $\eta \kappa_F(\bvarphi) < 1$, and the iterative scheme \cref{eq:iterative_scheme} is well-defined. Then, for any initial point $\bdf^0\in\mathbfcal{F}^D$, the iteration sequence $\{\bdf^k\}$ generated by \cref{alg:solve_cc_prob} converges to the unique solution of \cref{eq:cc_imaging_prob}.
\end{theorem}

\begin{proof}
Using \cref{eq:inverse_map}, the iterative scheme \cref{eq:iterative_scheme} can be rewritten as 
\begin{align}\label{eq:equal_cc_scheme}
    \Diff \bdK(\hat{\bdf}, \bdf^{k+1}-\bdf^k) = \bdg - \bdK(\bdf^k).
\end{align} 
The following estimate is obtained 
\begin{align*}
    \|\Diff \bdK(\hat{\bdf}, &\bdf^{k+1} - \bdf^*)\|_{\mathbfcal{G}^Q} \\
    &= \|\Diff \bdK(\hat{\bdf}, \bdf^{k+1}-\bdf^k) + \Diff \bdK(\hat{\bdf}, \bdf^{k}-\bdf^*)\|_{\mathbfcal{G}^Q} \\ 
    &\hspace{-1.4mm}\overset{\cref{eq:equal_cc_scheme}}{=} \|\bdK(\bdf^k) - \bdK(\bdf^*) - \Diff \bdK(\hat{\bdf}, \bdf^{k} - \bdf^*)\|_{\mathbfcal{G}^Q} \\ &=\Big\|\int_0^1 \Diff \bdK(t \bdf^k + (1-t)\bdf^*, \bdf^{k} - \bdf^*) -\Diff \bdK(\hat{\bdf},\bdf^{k}-\bdf^*)\rmd t\Big\|_{\mathbfcal{G}^Q} \\ 
    &\le \int_0^1 \|\Diff \bdK(t \bdf^k + (1-t)\bdf^*, \bdf^{k}-\bdf^*)-\Diff \bdK(\hat{\bdf}, \bdf^{k}-\bdf^*)\|_{\mathbfcal{G}^Q} \rmd t \\ 
    &\hspace{-1.4mm}\overset{\cref{eq:op_RGDC_0}}{\le} \eta \kappa_F(\bvarphi)\|\Diff \bdK(\hat{\bdf},\bdf^{k}-\bdf^*)\|_{\mathbfcal{G}^Q}.
\end{align*}
With $\eta \kappa_F(\bvarphi) < 1$, we have 
\begin{align*}
    \big\|\Diff \bdK(\hat{\bdf},\bdf^{k}-\bdf^*)\big\|_{\mathbfcal{G}^Q}\longrightarrow 0 \quad\text{as}~ k\longrightarrow \infty.
\end{align*}
Based on \cref{prop:0_deriva_invert}, we know that $\Diff \bdK(\hat{\bdf},\cdot)$ is invertible. Then, we obtain
\begin{equation*}
    \|\bdf^{k}-\bdf^*\|_{\mathbfcal{F}^D}\longrightarrow 0 \quad\text{as}~ k\longrightarrow \infty.
\end{equation*}
This completes the proof of the convergence. 

Finally, we prove the uniqueness of the solution. If there exists another solution $\bar{\bdf}^*$ satisfies \cref{eq:cc_imaging_prob}, according to the above estimate, we have 
\begin{equation*}
    \|\bdK(\bar{\bdf}^*) - \bdK(\bdf^*) - \Diff \bdK(\hat{\bdf},\bar{\bdf}^*-\bdf^*)\|_{\mathbfcal{G}^Q} \le  \eta \kappa_F(\bvarphi)\|\Diff \bdK(\hat{\bdf},\bar{\bdf}^*-\bdf^*)\|_{\mathbfcal{G}^Q}.
\end{equation*}
Thus, $\|\Diff \bdK(\hat{\bdf},\bar{\bdf}^*-\bdf^*)\|_{\mathbfcal{G}^Q}=0$, which implies that $\bar{\bdf}^*=\bdf^*$.
\end{proof}

\subsection{Convergence analysis for the discrete setting}

For the discrete case, the operator in \cref{eq:dd_imaging_prob} defined on Euclidean space is obviously continuously differentiable. Similarly, we can establish the following result.
\begin{lemma}\label{lem:discrete}
    Assume that $\bphi$ is invertible and $\mtbP$ has full column rank, furthermore, for any $\boldsymbol{\rho} \in \Real^M_{+}$ in the unit simplex, there exists a constant $\tilde{\eta} > 0$ such that for any indices $q$ and $d$,
    \begin{equation*}
        \big{|}  \mtbb_d^{\tra}\hat{\mtbw}^{[q]}  - \mtbb_d^{\tra} \boldsymbol{\rho} \big{|} < \tilde{\eta}  \big{|}  \mtbb_d^{\tra} \hat{\mtbw}^{[q]}\big{|}. 
    \end{equation*}
    Then, for any $\mtbf, \mtbh$, 
        \begin{equation*}%\label{eq:dd_RGDC}
        \|\nabla \mtbK(\hat{\mtbf})\mtbh - \nabla \mtbK(\mtbf)\mtbh\|_2\le \tilde{\eta}\kappa_F(\mtbP\otimes\bphi) \|\nabla \mtbK(\hat{\mtbf})\mtbh\|_2.
    \end{equation*}
\end{lemma}
\begin{proof}
    See \cref{proof:dd_lemma} for the proof.
\end{proof}

Considering the iterative scheme \cref{eq:dd_scheme_approx}, and using \cref{eq:dd_nabla_K0}, we denote a residual at each iteration by 
\begin{equation}\label{eq:dd_equal_scheme_error}
     \bddelta^k = \nabla \mtbK(\hat{\mtbf})(\mtbf^{k+1}-\mtbf^k) - \bigl(\mtbK(\mtbf^{k}) - \mtbg\bigr). 
\end{equation} 
For brevity, let 
\[
\gamma = \tilde{\eta}\kappa_F(\mtbP\otimes\bphi).
\] 
Under proper conditions, we have the following global convergence of the iterative scheme \cref{eq:dd_scheme_approx}.
\begin{theorem}\label{thm:dd_converge}
Assume that the conditions in \cref{lem:discrete} hold and $\gamma < 1$. Then, the solution to \cref{eq:dd_imaging_prob} is unique. 

Moreover, suppose that for any $k$, the residual in \cref{eq:dd_equal_scheme_error} satisfies 
\begin{equation}\label{eq:delta_k}
    \|\bddelta^k\|_2 \le  \zeta\|\mtbK(\mtbf^{k})-\mtbg\|_2,
\end{equation}
where the constant $\zeta < (1 - \gamma)/(1 + \gamma)$. Then, the sequence $\{\mtbf^{k}\}$ generated by the iterative scheme \cref{eq:dd_scheme_approx} converges to the unique solution of \cref{eq:dd_imaging_prob}.
\end{theorem}
\begin{proof}
    By the mean value theorem and \cref{lem:discrete}, for any $\mtbf$, we have the following estimate 
    \begin{align}\label{eq:dd_mean_value_esti}
        \|\mtbK(\mtbf)- \mtbK(\mtbf^*) - &\nabla \mtbK(\hat{\mtbf})(\mtbf-\mtbf^*)\|_2 \notag \\
        &= \Big\|\int_{0}^1 \nabla\mtbK\bigl(t\mtbf +(1-t)\mtbf^*\bigr)(\mtbf-\mtbf^*)-\nabla \mtbK(\hat{\mtbf})(\mtbf-\mtbf^*)\rmd t \Big\|_2 \notag \\ 
        &\le \int_{0}^1  \|\nabla\mtbK\bigl(t\mtbf +(1-t)\mtbf^*\bigr)(\mtbf-\mtbf^*) - \nabla \mtbK(\hat{\mtbf})(\mtbf-\mtbf^*)\|_2\rmd t\notag  \\ 
        &\le \gamma \|\nabla \mtbK(\hat{\mtbf})(\mtbf-\mtbf^*)\|_2.
    \end{align}
    
        Letting $\bar{\mtbf}^*$ be another solution if it exists, by \cref{eq:dd_mean_value_esti}, we obtain
    \begin{align*}
        \|\nabla \mtbK(\hat{\mtbf})(\bar{\mtbf}^*-\mtbf^*)\|_2
        \le \gamma \|\nabla \mtbK(\hat{\mtbf})(\bar{\mtbf}^*-\mtbf^*)\|_2.
\end{align*}
Using the assumption of $\gamma < 1$, it follows that
\begin{align*}
    \|\nabla \mtbK(\hat{\mtbf})(\bar{\mtbf}^*-\mtbf^*)\|_2=0.
\end{align*}
Since $\nabla \mtbK(\hat{\mtbf})$ has full column rank, we conclude that $\bar{\mtbf}^*=\mtbf^*$. Thus, the solution to \cref{eq:dd_imaging_prob} is unique.
    
    Furthermore, from \cref{eq:dd_mean_value_esti}, the triangle inequality implies that
   \begin{align*}
    \|\mtbK(\mtbf) - \mtbK(\mtbf^*)\|_2 \le (1 + \gamma) \|\nabla \mtbK(\hat{\mtbf})(\mtbf-\mtbf^*)\|_2.
   \end{align*}
   Hence, by \cref{eq:dd_equal_scheme_error,eq:delta_k,eq:dd_mean_value_esti}, and the above inequality in order, we have
   \begin{align}\label{eq:convergence_rate}
    \|\nabla \mtbK(\hat{\mtbf})(\mtbf^{k+1}-\mtbf^*)\|_2 &=  \|\nabla \mtbK(\hat{\mtbf})(\mtbf^{k+1}-\mtbf^k)+ \nabla \mtbK(\hat{\mtbf})(\mtbf^k-\mtbf^*)\|_2 \notag\\ 
    &\le \|\mtbK(\mtbf^k)-\mtbg-\nabla \mtbK(\hat{\mtbf})(\mtbf^k-\mtbf^*)\|_2 + \|\bddelta^k\|_2 \notag\\ 
    &\le \gamma \|\nabla \mtbK(\hat{\mtbf})(\mtbf^k-\mtbf^*)\|_2 + \zeta \|\mtbK(\mtbf^{k})-\mtbg\|_2 \notag\\ 
    &\le \bigl(\gamma + \zeta(1 + \gamma)\bigr)\|\nabla \mtbK(\hat{\mtbf})(\mtbf^{k}-\mtbf^*)\|_2. 
   \end{align}
   It follows from $\gamma + \zeta(1+\gamma)<1$ that the sequence $\{\|\nabla \mtbK(\hat{\mtbf})(\mtbf^k-\mtbf^*)\|_2\}$ converges to zero. Since $\nabla \mtbK(\hat{\mtbf})$ has full column rank, we obtain the convergence of the iteration sequence.
\end{proof}

Note that for the iterative scheme \cref{eq:dd_iterative_scheme}, the residual $\bddelta^k$ in \cref{eq:dd_equal_scheme_error} is always  vanishing. Following \cref{thm:dd_converge}, it is easy to obtain the convergence of the scheme  \cref{eq:dd_iterative_scheme}, namely,  \cref{alg:solve_dd_prob}. We present this result by the following corollary. 
\begin{corollary}\label{cor:dd_converge}
    Assume that the conditions in \cref{lem:discrete} hold, $\mtbP$ is invertible and $\gamma < 1$. Then, the sequence $\{\mtbf^{k}\}$ generated by the iterative scheme  \cref{eq:dd_iterative_scheme} converges to the unique solution of \cref{eq:dd_imaging_prob}.
\end{corollary}
\begin{proof}
The proof of \cref{thm:dd_converge} can be directly used here, so we omit the detail. 
\end{proof}  

\begin{remark}
    For the iterative scheme \cref{eq:dd_scheme_approx}, the estimate of \cref{eq:convergence_rate} presents a linear convergence in the sense of $\|\nabla \mtbK(\hat{\mtbf})(\mtbf^{k}-\mtbf^*)\|_2$. 
    Furthermore, by \cref{eq:convergence_rate}, we have that 
    \begin{align*}
    \sigma_{\min}\bigl(\nabla \mtbK(\hat{\mtbf})\bigr)\|\mtbf^{k+1}-\mtbf^*\|_2 \le&  \|\nabla \mtbK(\hat{\mtbf})(\mtbf^{k+1}-\mtbf^*)\|_2 \\
    \le& \bigl(\gamma + \zeta(1 + \gamma)\bigr)\|\nabla \mtbK(\hat{\mtbf})(\mtbf^{k}-\mtbf^*)\|_2\\ 
    \le& \sigma_{\max}\bigl(\nabla \mtbK(\hat{\mtbf})\bigr)\bigl(\gamma + \zeta(1 + \gamma)\bigr)\|\mtbf^{k}-\mtbf^*\|_2,  
\end{align*}
which implies  
\begin{equation}\label{eq:rate_convergence}
\|\mtbf^{k+1}-\mtbf^*\|_2  \le \kappa\bigl(\nabla \mtbK(\hat{\mtbf})\bigr) \bigl(\gamma + \zeta(1 + \gamma)\bigr)\|\mtbf^{k}-\mtbf^*\|_2.  
\end{equation}
Here $\kappa(\cdot) = \sigma_{\max}(\cdot)/\sigma_{\min}(\cdot)$ denotes the condition number. Hence, if $\kappa\bigl(\nabla \mtbK(\hat{\mtbf})\bigr) \bigl(\gamma + \zeta(1 + \gamma)\bigr) \in (0, 1)$, the Q-linear convergence of the proposed algorithm is established. Here the prefix ``Q" stands for ``quotient". 
    \end{remark}
    
    \begin{remark}\label{rem:trade_off}
  For the iterative scheme \cref{eq:dd_scheme_approx}, a more precise implementation of $\mtbP^{+}$ can generate a smaller residual $\bddelta^k$ in \cref{eq:dd_equal_scheme_error}. Then, there is a smaller $\zeta$ satisfying the condition of \cref{eq:delta_k}. From \cref{eq:rate_convergence}, it is easy to know that the smaller the $\zeta$, the faster the convergence rate of the algorithm. 
    \end{remark}

\section{Numerical experiments}\label{sec:Numerical}

We carry out numerical experiments to validate the effectiveness and numerical convergence of the proposed algorithm AFIRE for solving the image reconstruction problem in the geometric-inconsistent DECT in the case of mildly full scan. 

The compared algorithms include some existing methods such as the nonconvex primal-dual (NCPD) algorithm \cite{chpan21}, the nonlinear Kaczmarz method (NKM) using cyclic strategy \cite{gao23NKM}, and the IFBP reconstruction method \cite{zhao_19}. Moreover, we consider a practical algorithm that firstly processes the geometric-inconsistent data by using linear interpolation to obtain the geometric-consistent data with errors. Then the two-step DDD method is used to conduct the reconstruction. In the first step, the Newton's method can be used to estimate the basis sinograms by solving a lot of very small-scale independent nonlinear systems along different X-ray paths. In the second step, to reconstruct the basis images, the FBP method is applied to invert the finally obtained basis sinograms. To show the intermediate results, the FBP method is also utilized to reconstruct the basis images from the estimated basis sinograms at each Newton iteration in the first step. We refer to this algorithm as INTRPL for short.

\subsection{Experiment settings}

The tests are conducted on a workstation running Python equipped with an Intel i9-13900K @3.0GHz CPU and Nvidia RTX A4000 GPU. The experiment settings for the energy spectra and the MAC are presented as below. We simulate the discrete X-ray energy spectra $\{\mtbs^{[q]}\}_{q=1}^Q$ by an open-source X-ray spectrum simulator, SpectrumGUI \cite{spectrum_GUI}. Specifically, the used energy spectra pair, as displayed in \cref{fig:spectra_pair}, contains an 80-kV spectrum and a filtered 140-kV spectrum of the GE maxiray 125 X-ray tube, where the filter is a copper of 1-mm width. The MAC of the basis materials, namely, water and bone, are sourced from \cite{NIST_data} and used as the decomposition coefficients $\{\mtbb_d\}_{d=1}^D$. Notably, the use of two basis materials is a standard and physically well-justified practice in clinical DECT. Additionally, we employ two full-scan parallel-beam inconsistent geometries for low and high energy spectra, respectively. 
\begin{figure}[htbp]
    \centering
    \includegraphics[width=0.55\textwidth]{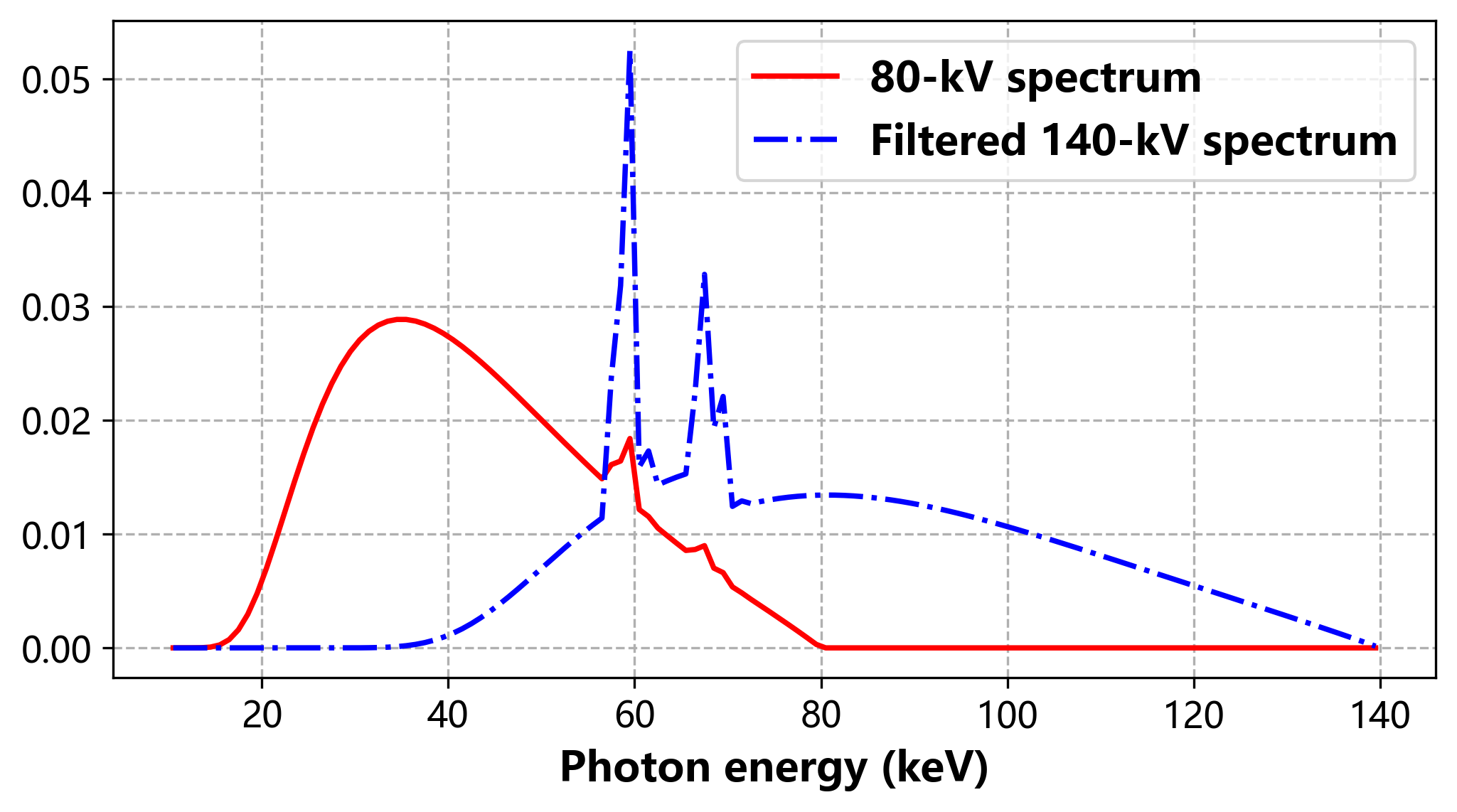}
    \caption{The used energy spectra pair in tests which contains an 80-kV spectrum (red, solid) and a filtered 140-kV spectrum (blue, dashed dot).}
    \label{fig:spectra_pair}
\end{figure}

With these configurations, we generate the simulated noiseless data by taking the spectra pair, MAC of water and bone, discrete X-ray transform and truth basis images into \cref{eq:discrete_imaging_op}. We utilize the following metrics of relative errors to assess the reconstruction accuracy 
\begin{align*}%\label{eq:REk_Data_fit}
    \text{RE}_{\mtbf}^{k} := \frac{\|\mtbf^{k} - \mtbf^*\|_2}{\|\mtbf^*\|_2},\quad \text{RE}_{\mtbg}^{k}:= \frac{\|\mtbK(\mtbf^{k}) - \mtbg\|_2}{\|\mtbg\|_2},
\end{align*}
where $\mtbf^{k}$ represents the $k$-th iteration point computed by the used algorithm, $\mtbf^*$ the given truth, and $\mtbg$ the simulated noiseless/noisy data. Moreover, we use the following error metrics between adjacent iterations to further evaluate the numerical convergence of the proposed algorithm 
\begin{align*}
    \Delta_{\mathbf{f}}^{k}:=\frac{\|\mathbf{f}^{k+1}-\mathbf{f}^{k}\|_2}{\|\mathbf{f}^{k}\|_2}, \quad \Delta_{\mathbf{g}}^k:=\frac{\|\mathbf{K}(\mathbf{f}^{k+1})-\mathbf{K}(\mathbf{f}^{k})\|_2}{\|\mathbf{g}\|_2}.
\end{align*} 

In \cref{alg:solve_dd_prob} of AFIRE, as stated in \cref{rem:dd_special_point,rem:inverse_comm}, we take $\hat{\mtbf}=\bzero$, and implement the $\mtbP_{q}^{+}$, $q = 1, \ldots, Q$ by FBP method where the 1-bandwidth Ram--Lak filter is used. Moreover, INTRPL and IFBP also use the same FBP method, and NCPD is used to solve a total variation (TV)-regularized nonlinear least square problem with nonnegative constraint. For NCPD, the regularization parameter is selected as $0$ and $10^{-6}$ for the noiseless and noisy data, respectively. Without loss of generality, the initial point for all tested algorithms are set to the zero vector.

\subsection{Simulated noiseless data for the forbild phantom}
\label{subsec:noiseless_forbild}

In this test, a forbild phantom is composed of the truth water and bone basis images which are both discretized as $128\!\times\!128$ pixels on the square area $[-5,5]\!\times\![-5,5]$ $\text{cm}^2$ and take gray values over $[0, 1]$, as shown in row 1 of \cref{fig:recon_forbild}. 

The X-ray paths under the different spectra displayed in \cref{fig:spectra_pair} are completely  inconsistent. Specifically, for the 80-kV spectrum, we utilize 384 parallel projections uniformly sampled on $[-7.05, 7.05]$ cm for each of the 384 views, where these views are uniformly distributed over the interval $[0, \pi)$ and can be seen as a mildly full scan. Under the filtered 140-kV spectrum, the scanning also has the same number of projections and views, but the views are uniformly distributed over the interval $[\pi/768, \pi/768+\pi)$.

After 100 iterations, we plot the curves of metrics $\text{RE}_{\mtbf}^{k}$ and $\text{RE}_{\mtbg}^{k}$ in \cref{fig:REk_forbild}. 
\begin{figure}[htbp]
    \centering
    \includegraphics[width=0.49\textwidth]{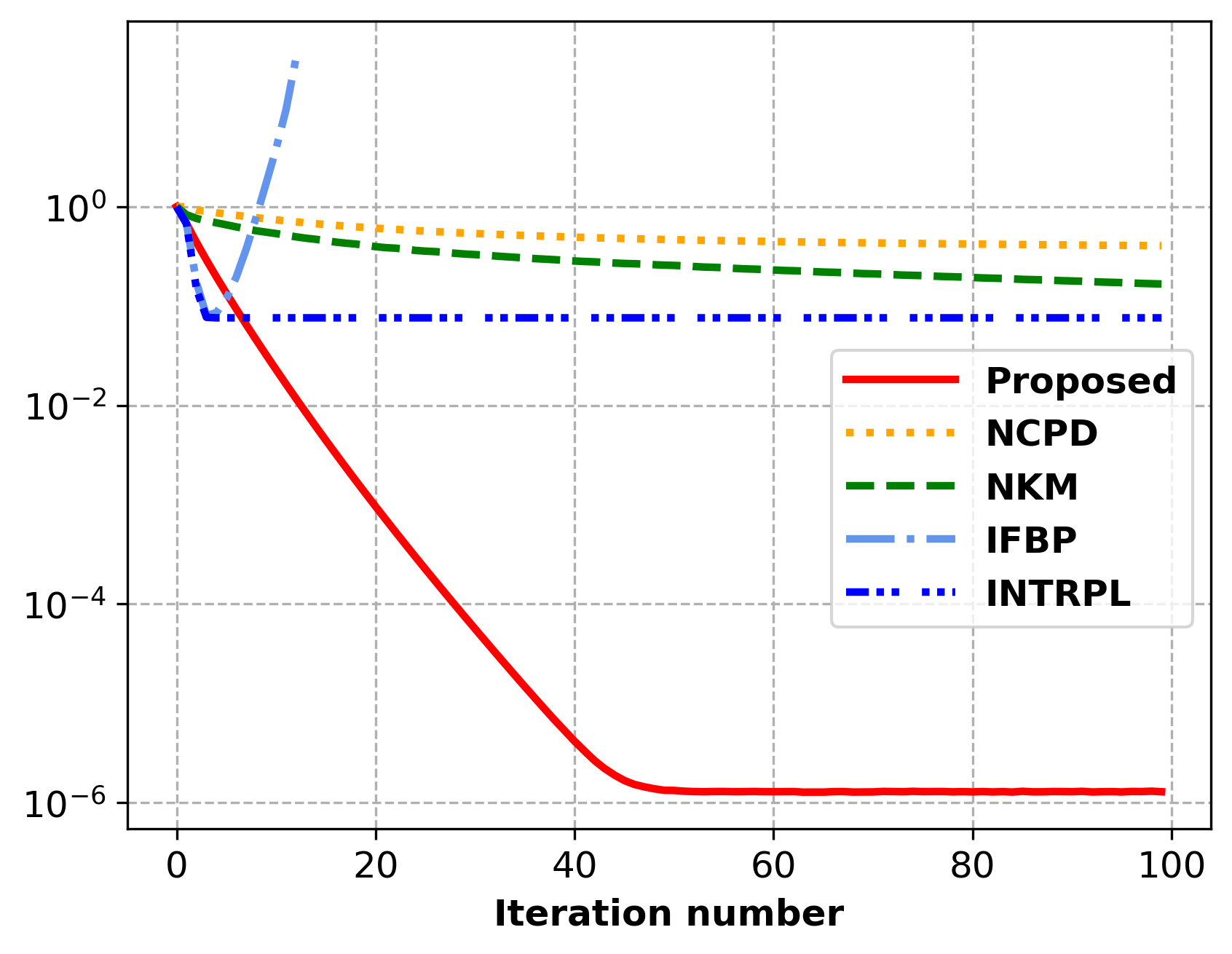}
    \includegraphics[width=0.49\textwidth]{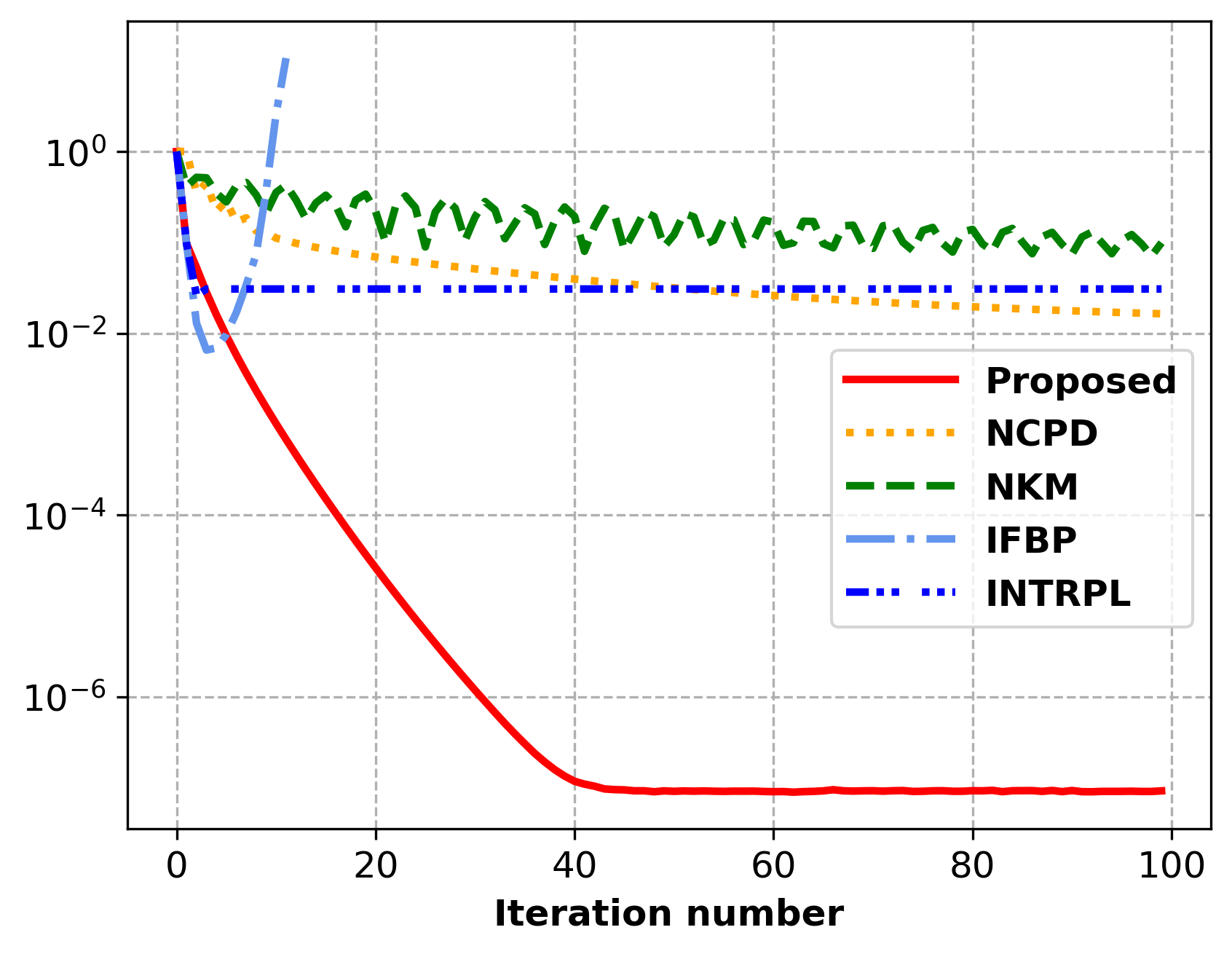}
    \caption{Metrics {\rm RE$_{\mtbf}^{k}$} (left) and {\rm RE$_{\mtbg}^{k}$} (right) are plotted in semi-log scale as functions of iteration numbers for reconstructing basis images of the forbild phantom by the proposed algorithm AFIRE, and NCPD, NKM, IFBP, INTRPL from the noiseless data.}
    \label{fig:REk_forbild}
\end{figure} 
It is shown by the metrics of relative error that the proposed algorithm AFIRE has reached numerical convergence and reconstructed high-accuracy result in less than 50 iterations, thereby demonstrating its effectiveness and accuracy for inverting noiseless data. The results also confirm the convergence theory established in \cref{thm:dd_converge}. In contrast, NCPD and NKM are still far from convergence, and INTRPL can only produce a low-accuracy result. Particularly, it is observed that IFBP does not converge. This is because IFBP, as a heuristic approach that designed for geometric-inconsistent case, lacks rigorous theoretical guarantees. 

The metrics of relative errors $\text{RE}_{\mtbf}^{k}$ and $\text{RE}_{\mtbg}^{k}$ over time are also presented in \cref{fig:test1_time}. We observe that the proposed algorithm AFIRE exhibits superior performance, demonstrating its advantages in both efficiency and accuracy over the other compared algorithms. 
\begin{figure}[htbp]
    \centering
    \includegraphics[width=0.49\textwidth]{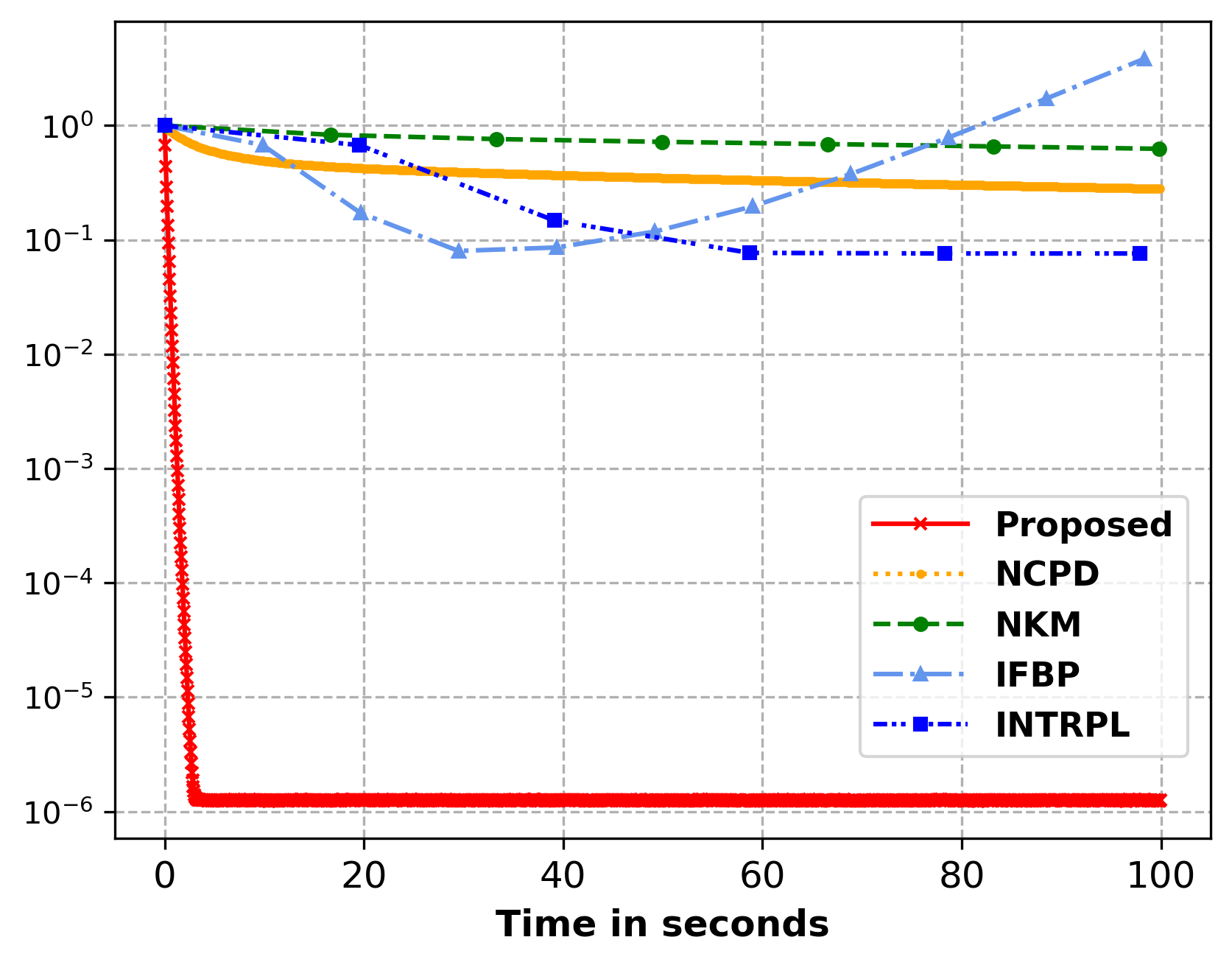}
    \includegraphics[width=0.49\textwidth]{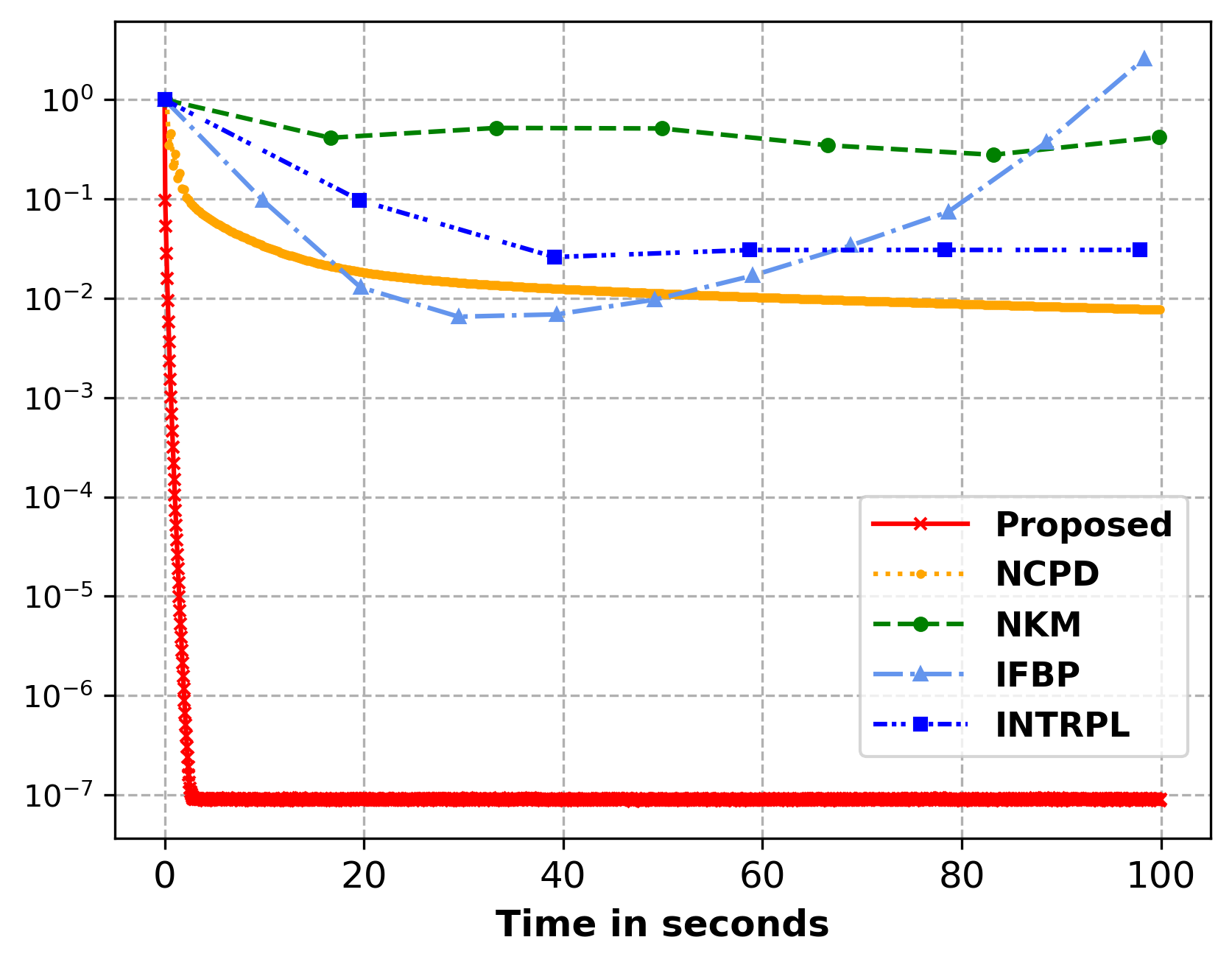}
    \caption{Metrics {\rm RE$_{\mtbf}^{k}$} (left) and {\rm RE$_{\mtbg}^{k}$} (right) are plotted in semi-log scale over time (in seconds) for reconstructing basis images of the forbild phantom by the proposed algorithm AFIRE, and NCPD, NKM, IFBP, INTRPL from the noiseless data.}
    \label{fig:test1_time}
\end{figure}  

We further plot the curves of metrics $\Delta_{\mathbf{f}}^{k}$ and $\Delta_{\mathbf{g}}^k$ in \cref{fig:conver_metric_forbild}, which also demonstrate the numerical convergence of the proposed algorithm AFIRE. 
\begin{figure}[htbp]
    \centering
    \includegraphics[width=0.49\textwidth]{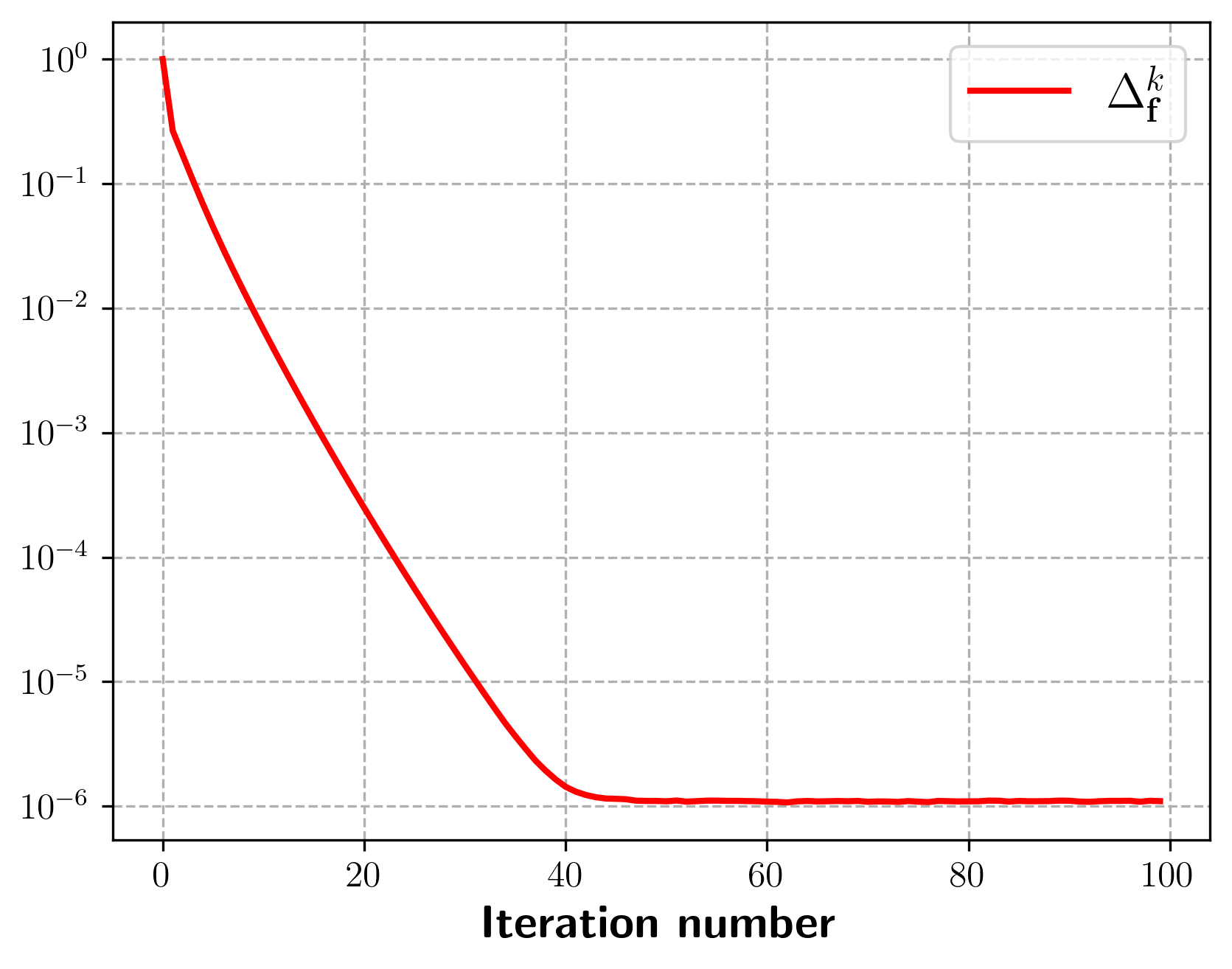}
    \includegraphics[width=0.49\textwidth]{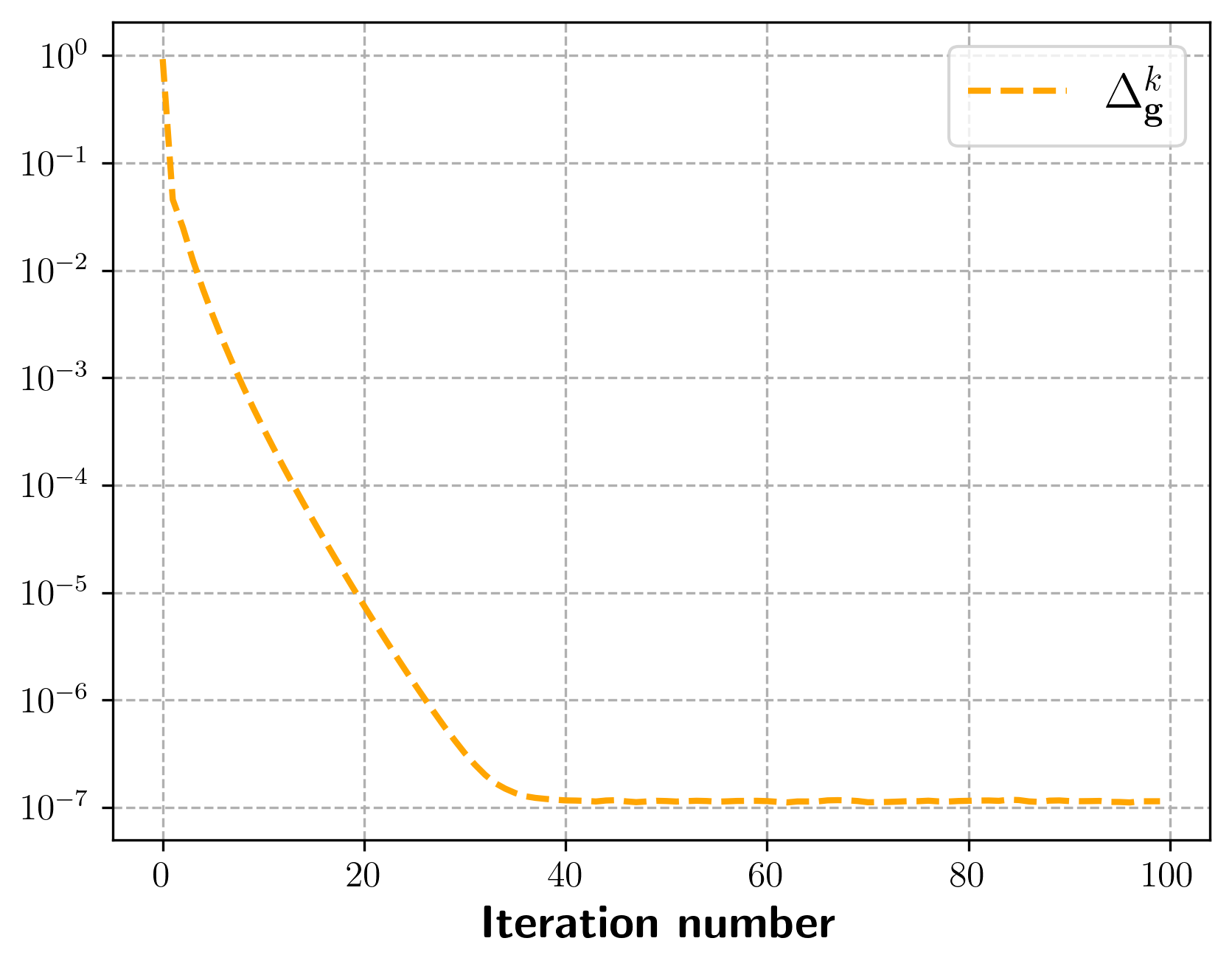}
    \caption{Metrics $\Delta_{\mathbf{f}}^{k}$ (left) and $\Delta_{\mathbf{g}}^k$ (right) are plotted in semi-log scale as functions of iteration numbers for reconstructing basis images of the forbild phantom by the proposed algorithm AFIRE from the noiseless data.}
    \label{fig:conver_metric_forbild}
\end{figure} 

Furthermore, as shown in \cref{fig:recon_forbild}, the reconstructed basis images and the obtained VMIs after 100 iterations are displayed for all the above algorithms except for IFBP, which reveal that the reconstructed results by the proposed algorithm AFIRE are closest to the truths.  
\begin{figure}[htbp]
    \centering
    \includegraphics[width=0.9\textwidth]{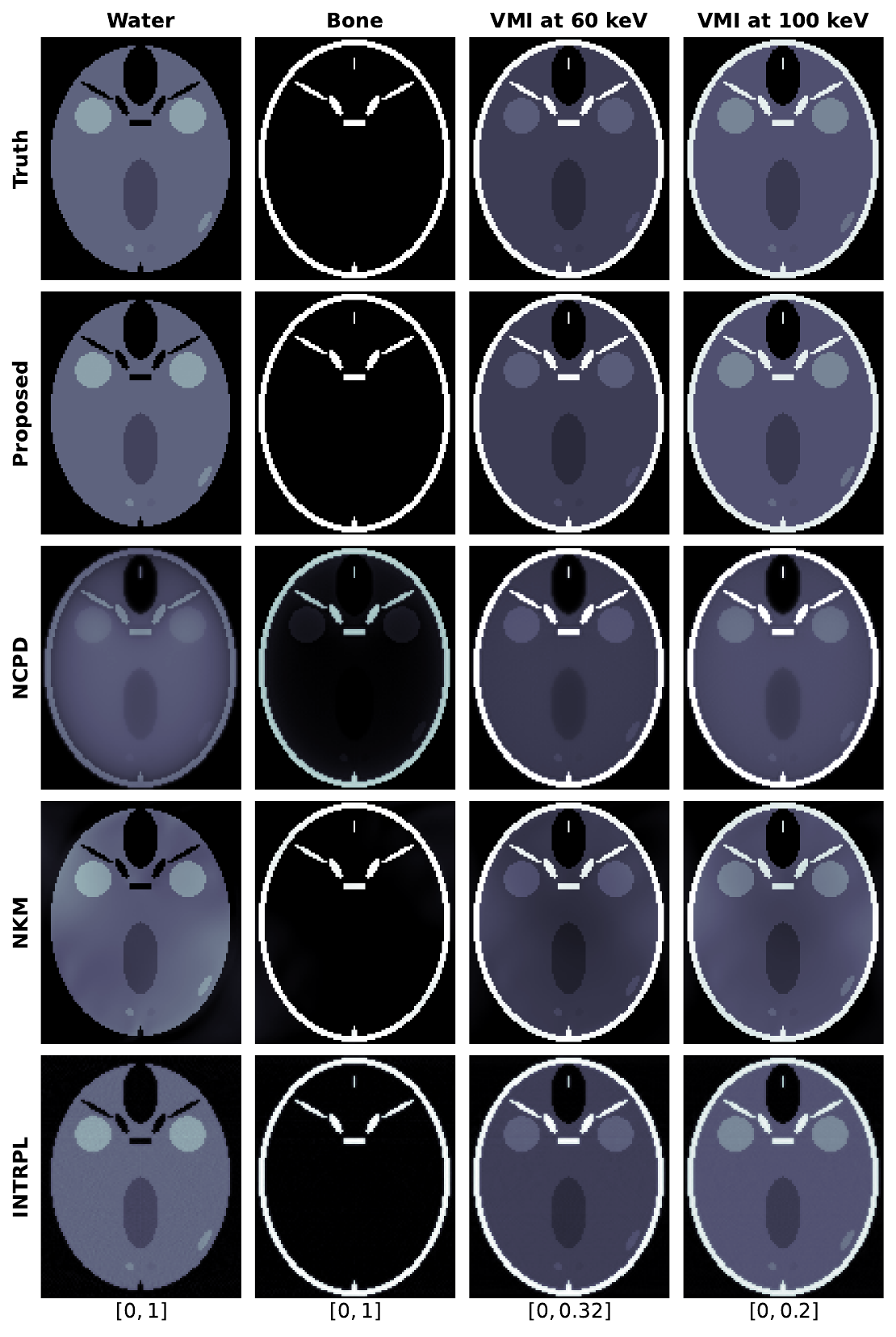}
    \caption{From left to right: basis images of water and bone, VMIs at energies 60 keV and 100 keV  of the forbild phantom. From top to bottom: the truths, the results after 100 iterations using the proposed algorithm AFIRE, and NCPD, NKM, INTRPL.}
    \label{fig:recon_forbild}
\end{figure} 

\subsection{Simulated noisy data for the realistic torso image}\label{subsec:noisy_torso}

For this test, a real patient torso image is decomposed into the truth water and bone basis images which are both discretized as $256\!\times\!256$ pixels on the square area $[-5,5]\!\times\![-5,5]$ $\text{cm}^2$ and take gray values over $[0, 1]$, as shown in row 1 of \cref{fig:recon_torso}.

Similarly, the X-ray paths under the different spectra shown in \cref{fig:spectra_pair} are also completely inconsistent. More precisely, under the 80-kV spectrum, 768 parallel-beam projections uniformly distributed on $[-7.05, 7.05]$ cm are taken for each of 768 views, where these views are evenly sampled on the interval $[0, \pi)$. As to the filtered 140-kV spectrum, the same numbers of projections and views are utilized, but the views are uniformly distributed over the interval $[\pi/1536, \pi/1536+\pi)$. Moreover, the 34.3 dB noisy data is generated by adding Gaussian noise onto the noiseless data. 

We plot the curves of metrics $\text{RE}_{\mtbf}^{k}$ and $\text{RE}_{\mtbg}^{k}$ in \cref{fig:REk_torso} within 100 iterations. It is obvious that the metrics of relative errors for the proposed algorithm AFIRE exhibit numerical convergence and relatively high accuracy in less than $15$ iterations, demonstrating its effectiveness and stability for addressing noisy data. On the contrary, NCPD and NKM have not converged yet; INTRPL can only produce a low-accuracy result; IFBP is divergent after a couple of iterations.  
\begin{figure}[htbp]
    \centering
    \includegraphics[width=0.49\textwidth]{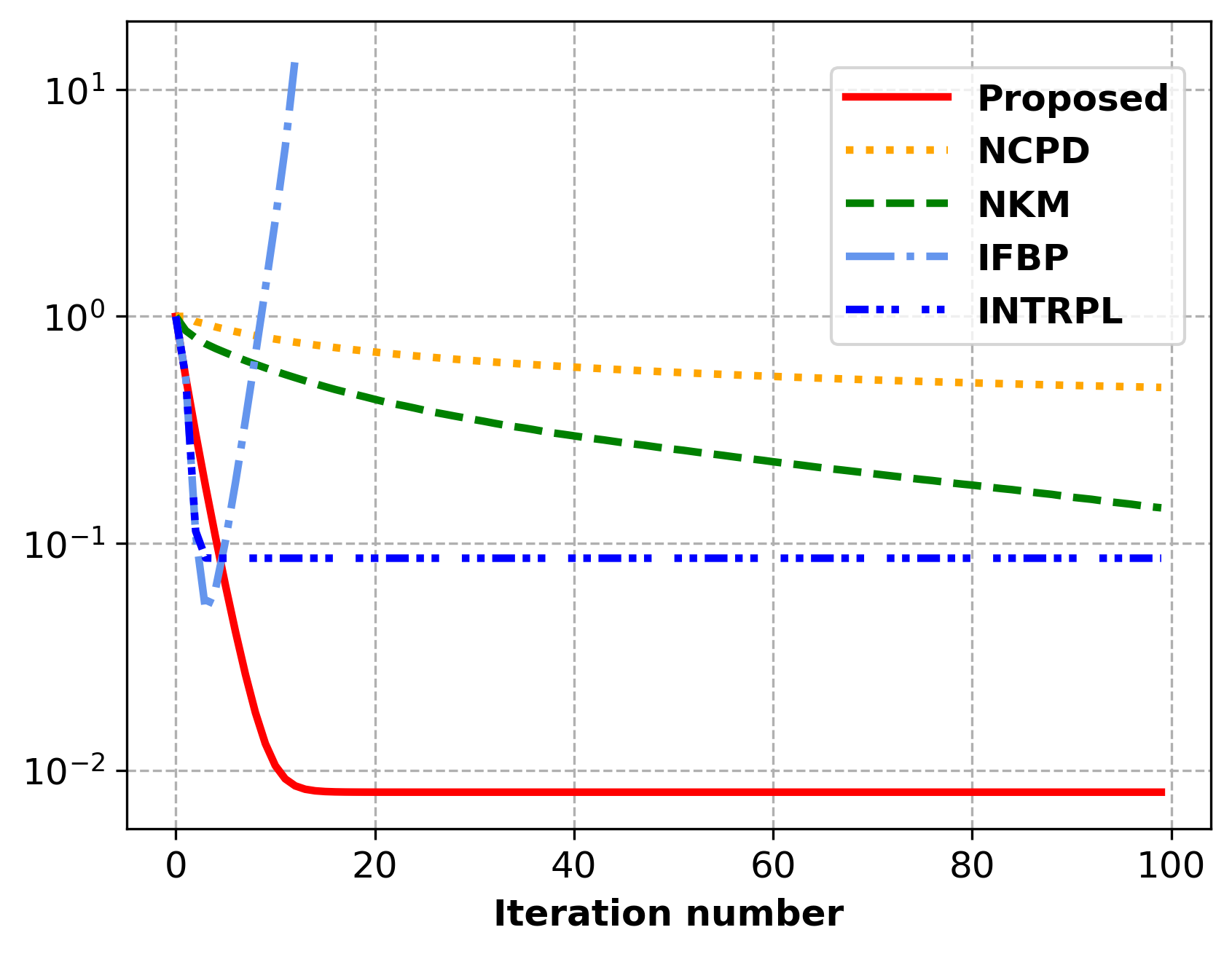}
    \includegraphics[width=0.49\textwidth]{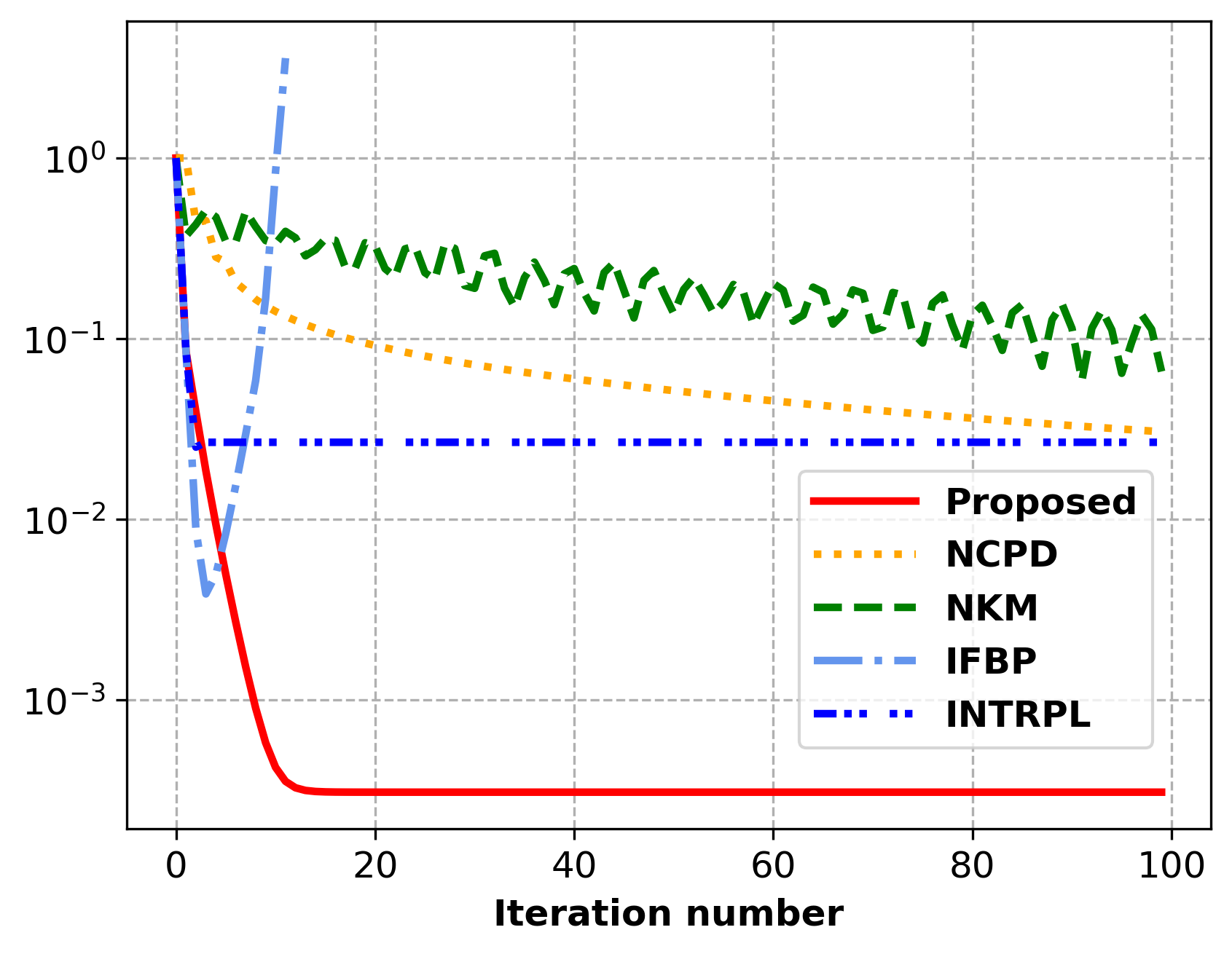}
    \caption{Metrics {\rm RE$_{\mtbf}^{k}$} (left) and {\rm RE$_{\mtbg}^{k}$} (right) are plotted in semi-log scale as functions of iteration numbers for reconstructing basis images of the real torso image by  the proposed algorithm AFIRE, and NCPD, NKM, IFBP, INTRPL from the noisy data.}
    \label{fig:REk_torso}
\end{figure} 

The metrics of relative errors $\text{RE}_{\mtbf}^{k}$ and $\text{RE}_{\mtbg}^{k}$ over time are presented in \cref{fig:test2_time}. It is observed that the proposed algorithm AFIRE exhibits superior performance, demonstrating its advantages in both efficiency and accuracy over the other compared algorithms. 
\begin{figure}[htbp]
    \centering
    \includegraphics[width=0.49\textwidth]{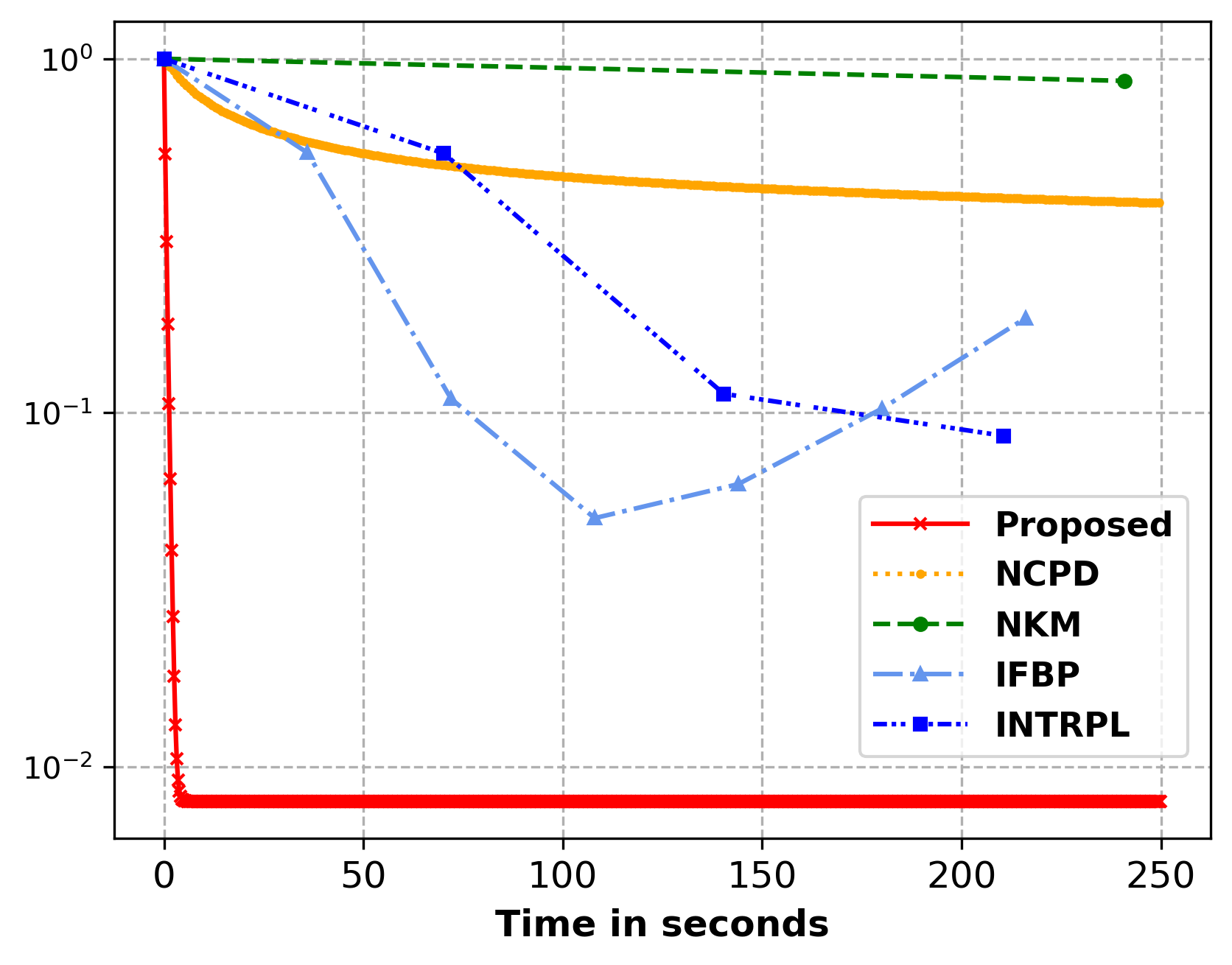}
    \includegraphics[width=0.49\textwidth]{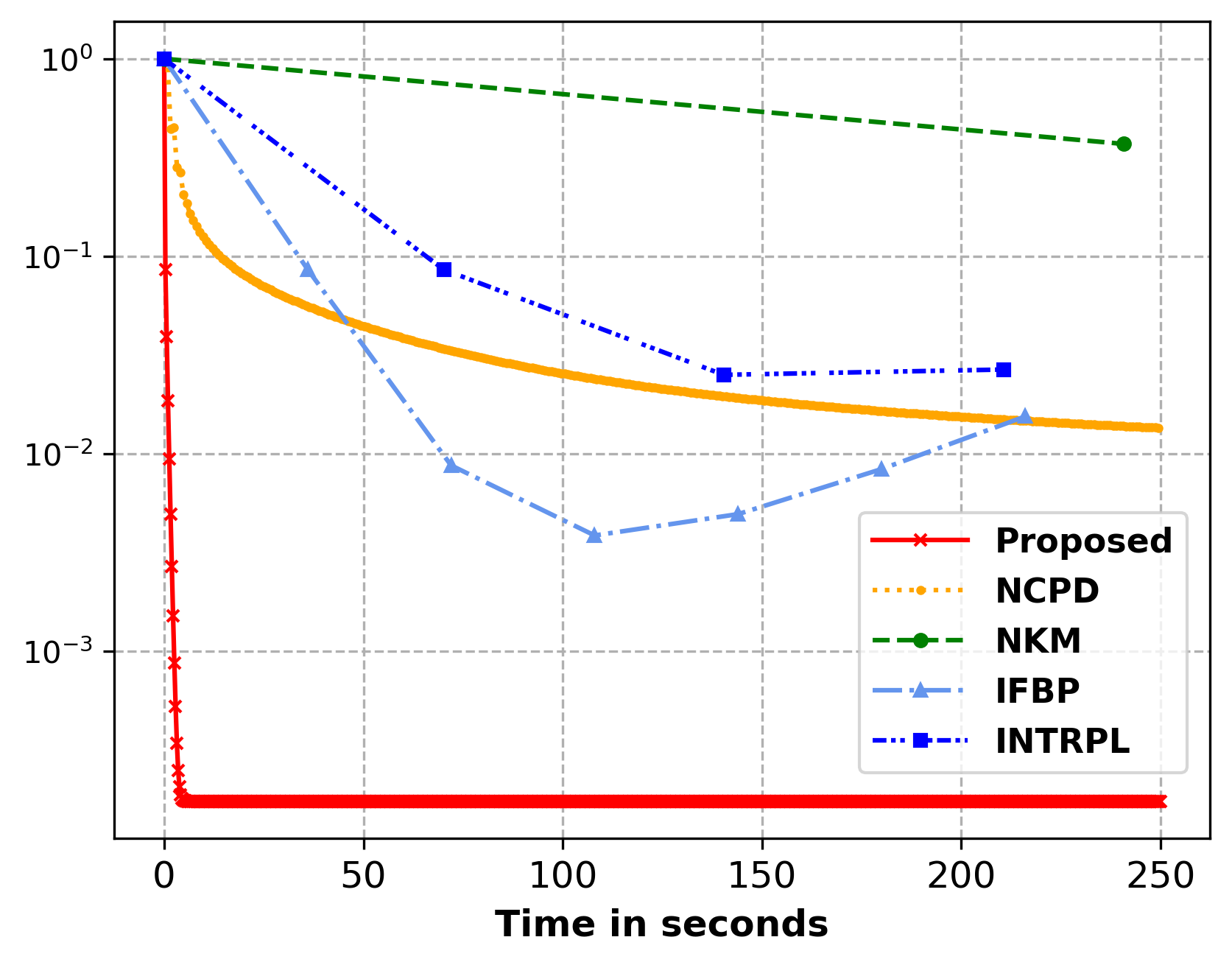}
    \caption{Metrics {\rm RE$_{\mtbf}^{k}$} (left) and {\rm RE$_{\mtbg}^{k}$} (right) are plotted in semi-log scale over time (in seconds) for reconstructing basis images of the real torso image by the proposed algorithm AFIRE, and NCPD, NKM, IFBP, INTRPL from the noisy data.}
    \label{fig:test2_time}
\end{figure} 
We also present the curves of metrics $\Delta_{\mathbf{f}}^{k}$ and $\Delta_{\mathbf{g}}^k$ in \cref{fig:conver_metric_torso}, which validate the numerical convergence of the proposed algorithm AFIRE. 
\begin{figure}[htbp]
    \centering
    \includegraphics[width=0.49\textwidth]{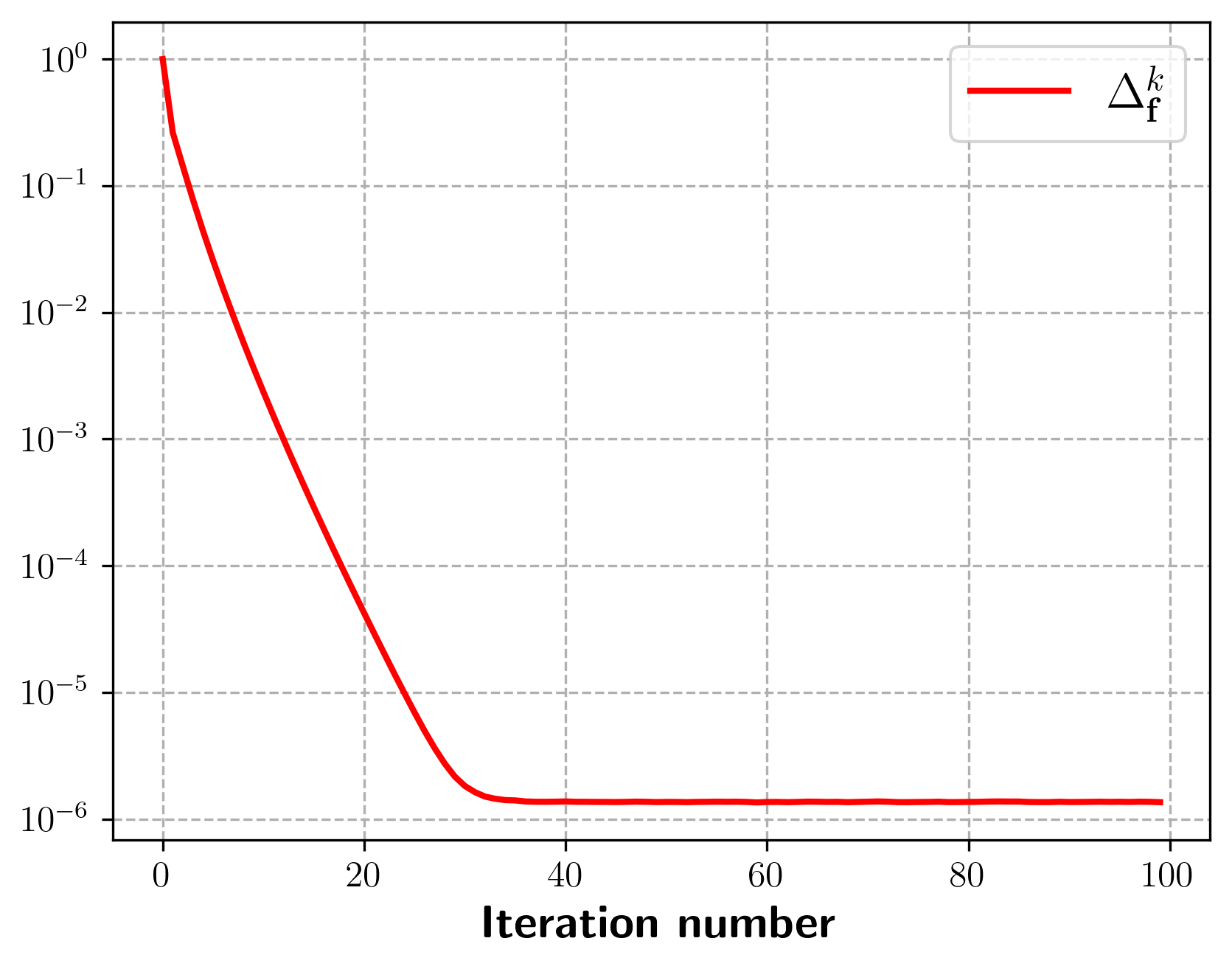}
    \includegraphics[width=0.49\textwidth]{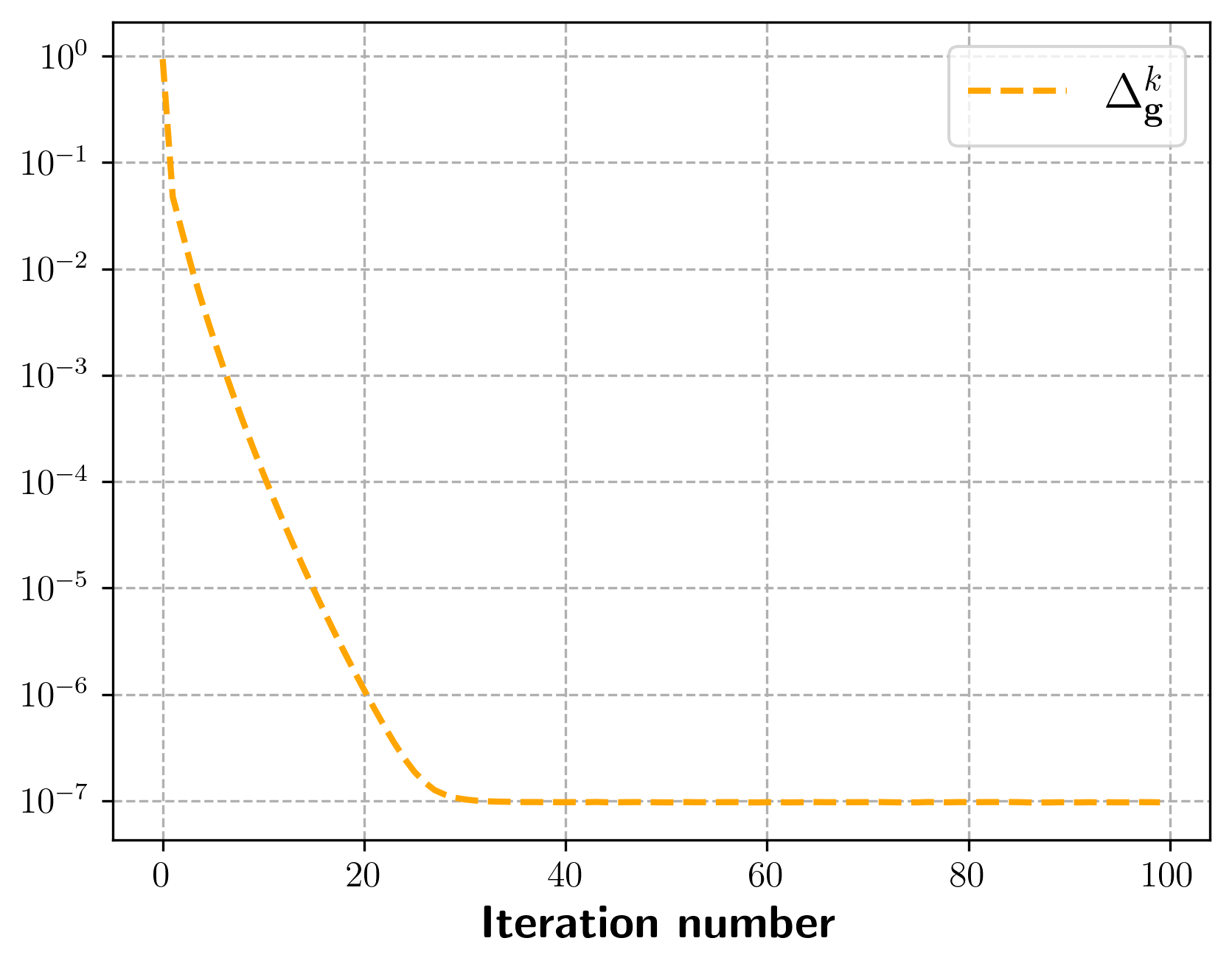}
    \caption{Metrics $\Delta_{\mathbf{f}}^{k}$ (left) and $\Delta_{\mathbf{g}}^k$ (right) are plotted in semi-log scale as functions of iteration numbers for reconstructing basis images of the real torso image by the proposed algorithm AFIRE from the noisy data.}
    \label{fig:conver_metric_torso}
\end{figure}

Moreover, after 100 iterations, the reconstructed basis images and the obtained VMIs are displayed in \cref{fig:recon_torso} for all the above algorithms except for IFBP, which reveal that the reconstructed images by the proposed algorithm AFIRE are closest to the truths. 
\begin{figure}[htbp]
    \centering
    \includegraphics[width=1.\textwidth]{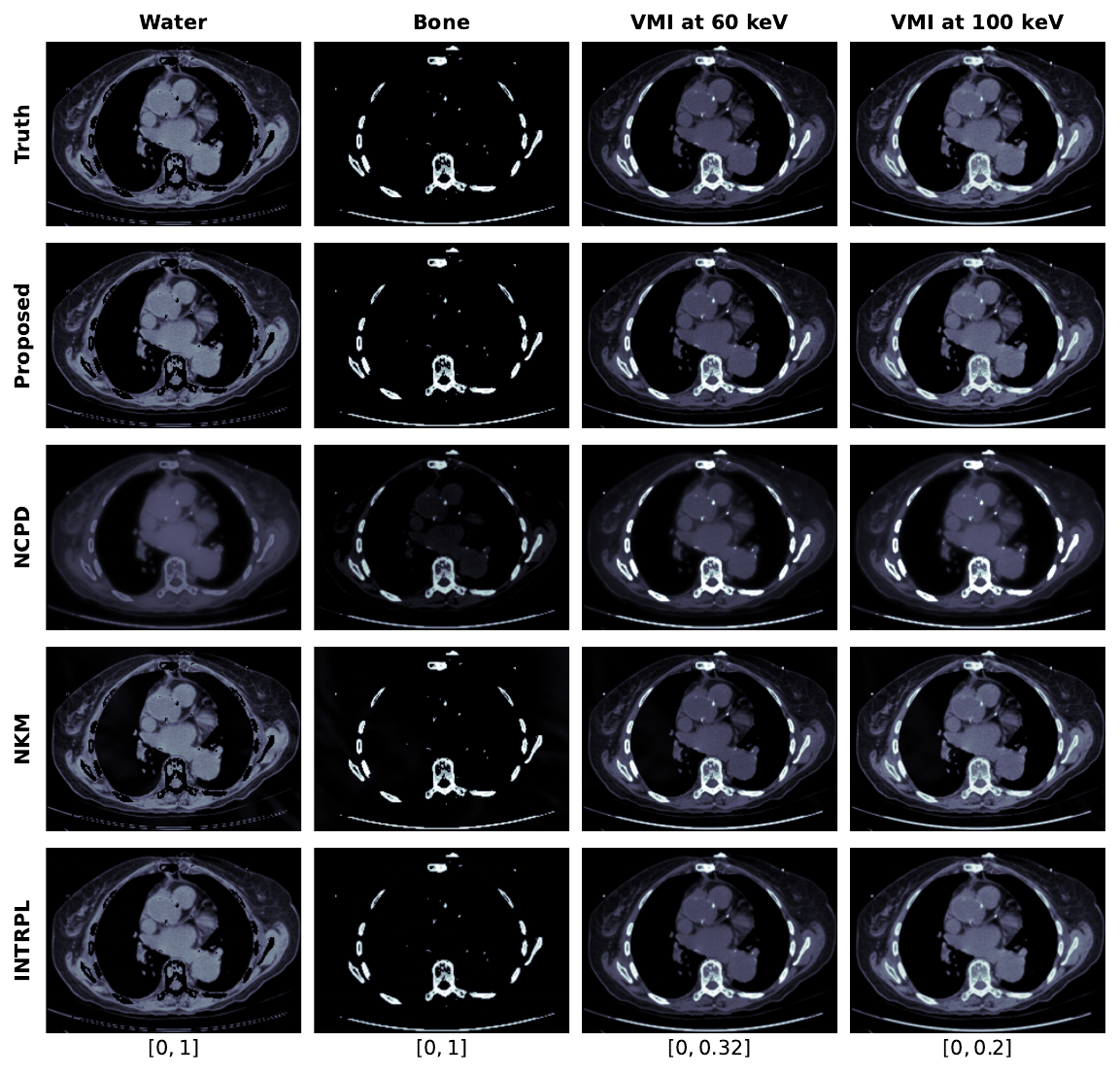}
    \vspace{-5mm}
    \caption{From left to right: basis images of water and bone, VMIs at energies 60 keV and 100 keV of the realistic torso image. From top to bottom: the truths, the results after 100 iterations using the proposed algorithm AFIRE, and NCPD, NKM, INTRPL.}
    \label{fig:recon_torso}
\end{figure}

\section{Discussion}
\label{sec:discussion}

In this section, we discuss more details about the proposed algorithm AFIRE. Actually, this algorithm can be seen as a framework for solving the image reconstruction problem in MSCT, since there are two main adjustable modules, for instance, including the selection of the special point $\hat{\mtbf}$ and the implementation of the approximated inverse $\mtbP^{+}$ in \cref{alg:solve_dd_prob} of AFIRE. On the other hand, although this algorithm was originally designed to handle the reconstruction problem of geometric inconsistency between low and high X-ray energy spectra,  it can be directly used to solve the problem in geometric-consistent situation. Hence, the flexibility and extensibility may enhance the applicability of the proposed algorithm. 

To begin with, we introduce some basic experiment settings. The same energy spectra and MAC are applied as those described in \cref{sec:Numerical}. Differently, the tested water and bone basis images are originated from the realistic CT images which are randomly selected from the dataset of LoDoPaB-CT \cite{leuschner2021lodopab}, as shown in row 1 of \cref{fig:discuss_recon_hatf,fig:discuss_recon_P_inv,fig:discuss_recon_geometry_consis}. All the basis images are discretized as $362\!\times\!362$ pixels on the square area $[-15,15]\!\times\![-15,15]$ $\text{cm}^2$. 

In the discussion of the two adjustable modules, under the 80-kV spectrum, we employ 1086 parallel projections evenly sampled on $[-21.15, 21.15]$ cm for each of the 900 views, where these views are uniformly distributed over the interval $[0, \pi)$. Under the filtered 140-kV spectrum, the scanning uses the same numbers of projections and views, but the views are uniformly distributed over the interval $[\pi/1800, \pi/1800+\pi)$. With these configurations, we generate simulated noiseless data as before and apply \cref{alg:solve_dd_prob} with selecting  different modules to solve the corresponding reconstruction problems.

\subsection{Selection of the special point $\hat{\mtbf}$}
\label{dis:hatf}

Based on our insights in \cref{sec:propose_alg}, in the discrete (continuous) case, obviously, for the real constant set $\{C_d^{[q]}=0\}_{d,q=1}^{D,Q}$ ($\{C_d = 0\}_{d=1}^D$), there exists $\hat{\mtbf} = \bzero$ ($\hat{\bdf} = \bzero$) fulfilling the condition \cref{eq:dd_special_points} (the condition \cref{eq:cc_special_points}). As we can imagine, those are almost the simplest examples that satisfying the required conditions. 

The results in \cref{sec:Numerical} have demonstrated the effectiveness and numerical convergence of the proposed algorithm when we set each $C_d^{[q]} = 0$ with $\hat{\mtbf}=\bzero$. Besides that, we further investigate alternative choices. Notably, the implementation of \cref{alg:solve_dd_prob} depends only on the values of certain constant set $\{C_d^{[q]}\}_{d,q=1}^{D,Q}$, even if it is unknown whether there exists an $\hat{\mtbf}$ to satisfy the condition \cref{eq:dd_special_points}. 

To discuss these alternatives, we just generate three random arrays over [0,1] as different selections of $\{C_d^{[q]}\}_{d,q=1}^{D,Q}$ using \texttt{numpy.random} with different seeds, but ignore searching for the concrete special point $\hat{\mtbf}$.  Additionally, all the $\mtbP_{q}^{+}$, $q=1, \ldots, Q$, are implemented by the FBP with 1-bandwidth Ram--Lak filter. 

As shown in \cref{fig:discuss_RE_choice_hatf}, we plot the curves of relative errors $\text{RE}_{\mtbf}^{k}$ and $\text{RE}_{\mtbg}^{k}$ in $100$ iterations. 
\begin{figure}[htbp]
    \centering
    \includegraphics[width=0.49\textwidth]{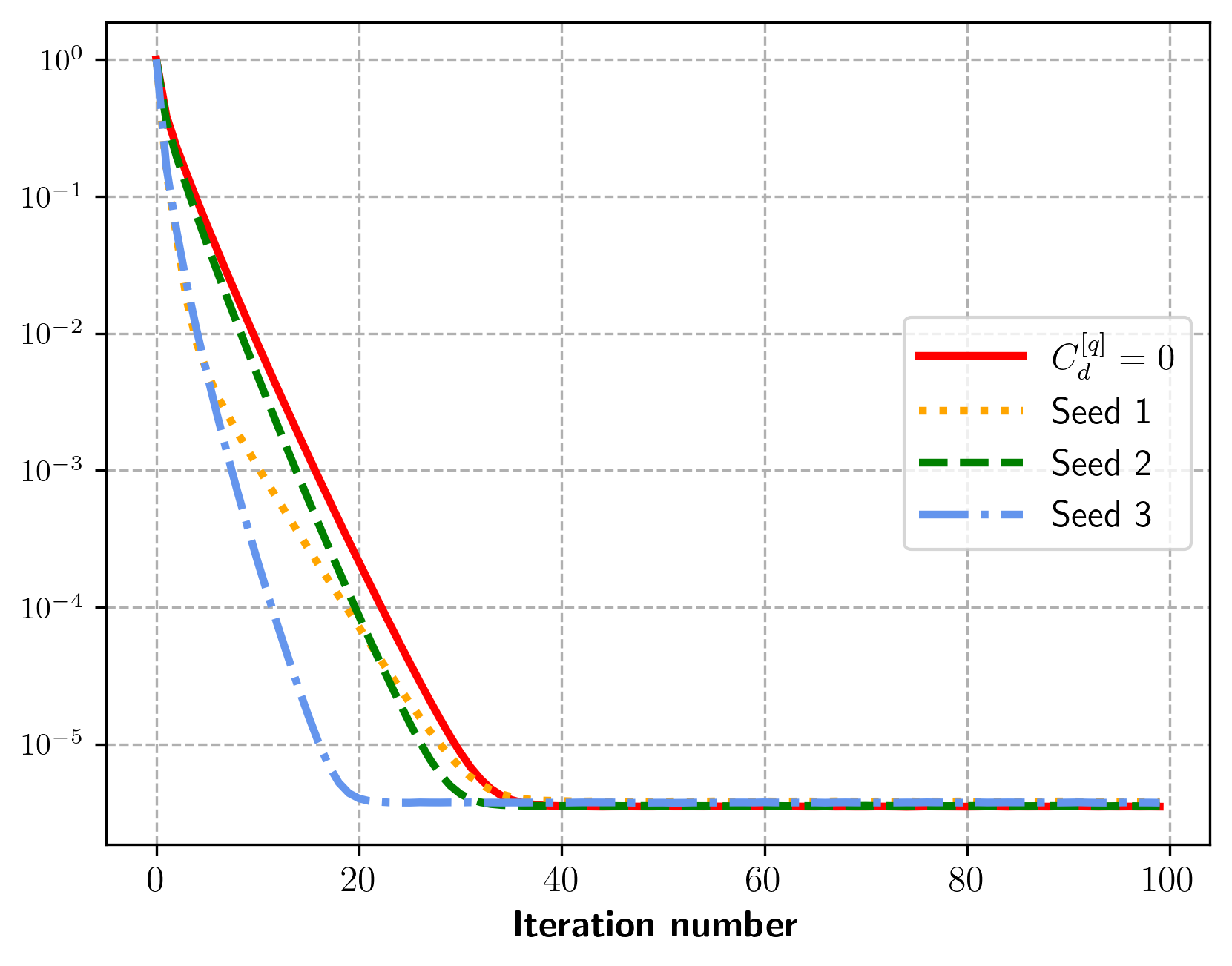}
    \includegraphics[width=0.49\textwidth]{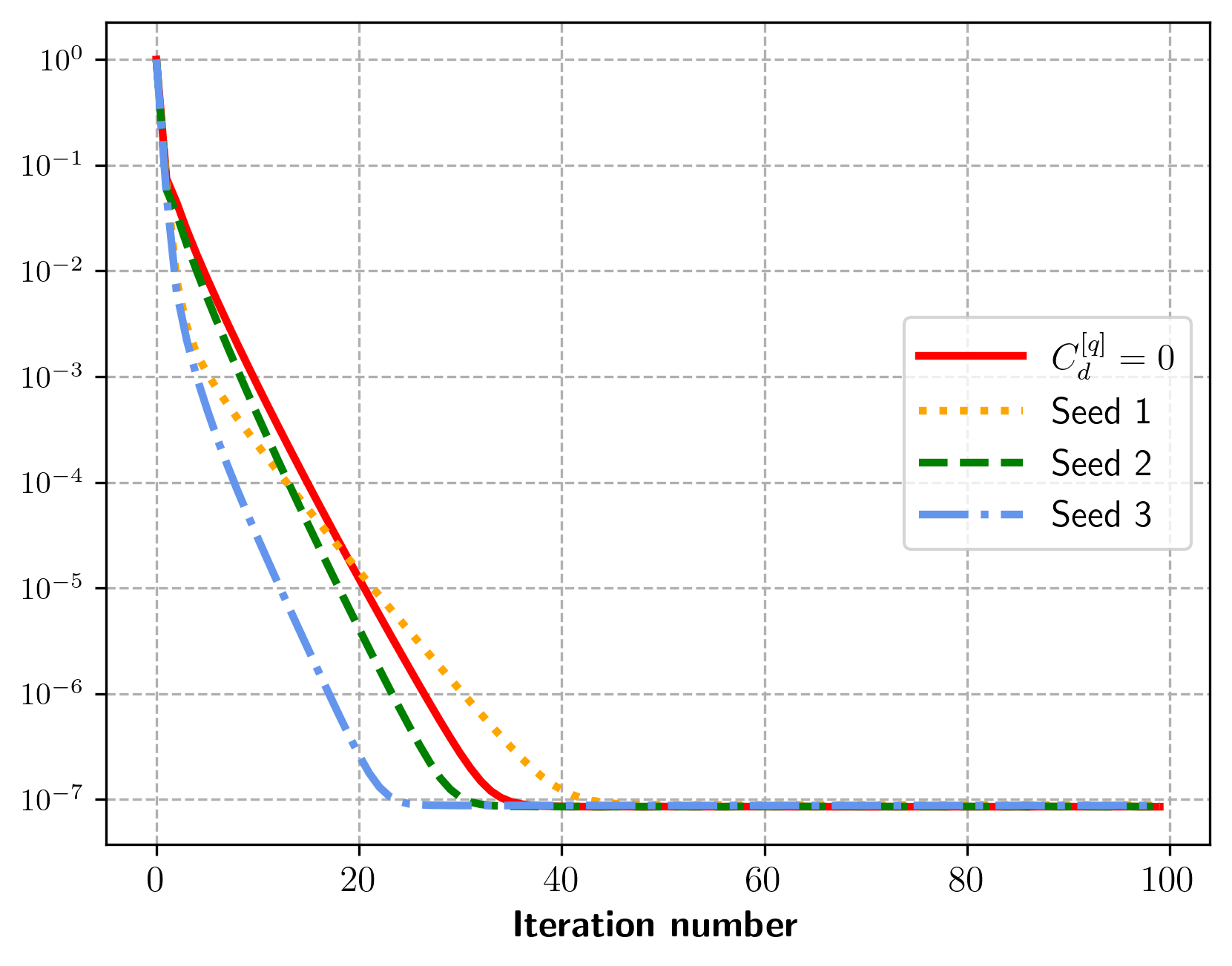}
    \caption{Metrics {\rm RE$_{\mtbf}^{k}$} (left) and {\rm RE$_{\mtbg}^{k}$} (right) are plotted in semi-log scale as functions of iteration numbers for reconstructing basis images of a clinical image in \cref{fig:discuss_recon_hatf} by the proposed algorithm AFIRE with different selections of $\{C_d^{[q]}\}_{d,q=1}^{D,Q}$.}
    \label{fig:discuss_RE_choice_hatf}
\end{figure}
It is discovered that the convergence behaviors of the proposed algorithm using these three different arrays are almost the same as that using the vanishing one. The basis images and VMIs obtained by the proposed algorithm using these arrays are displayed in \cref{fig:discuss_recon_hatf}, and they are visually comparable and are all quite close to the truth images. 
\begin{figure}[htbp]
    \centering
    \includegraphics[width=1.\textwidth]{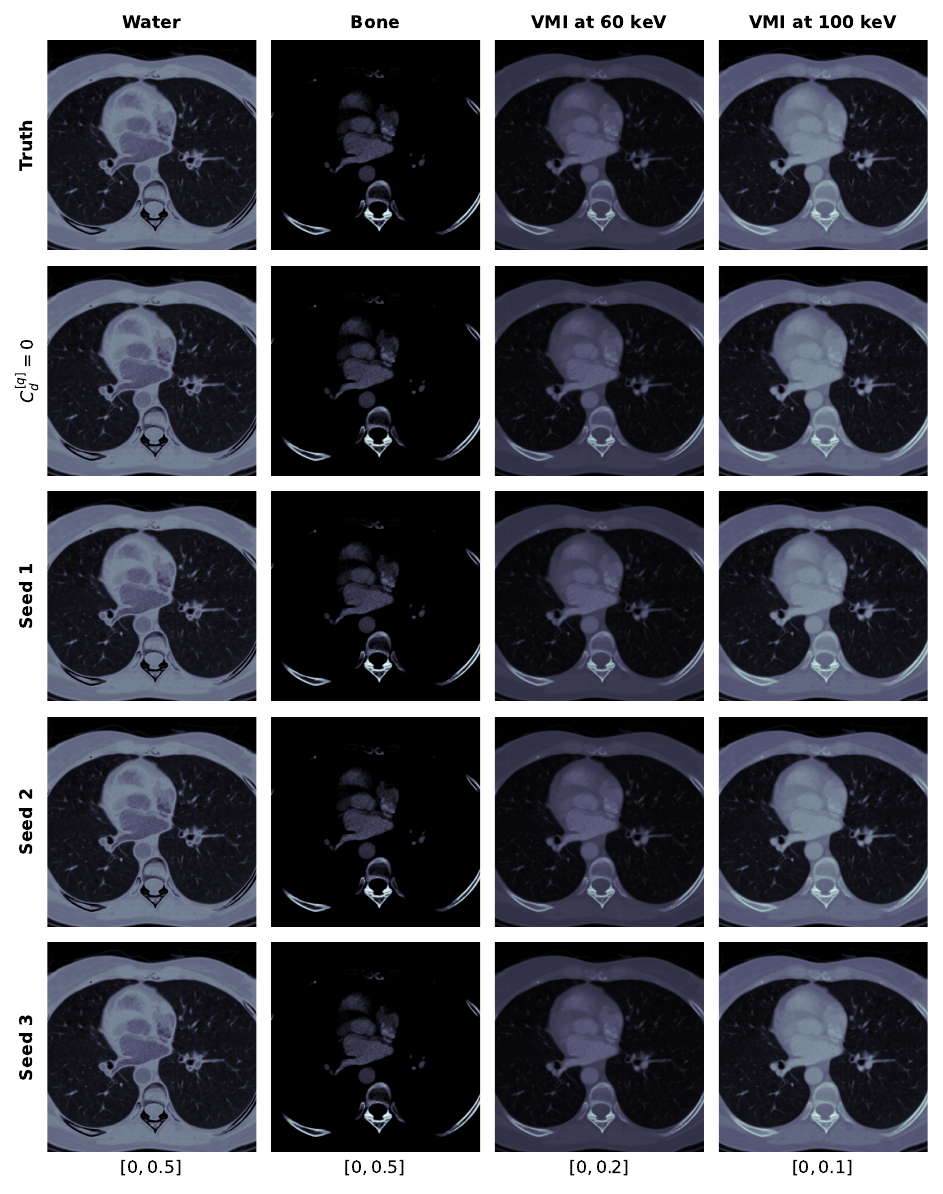}
    \vspace{-5mm}
    \caption{From left to right: basis images of water and bone, VMIs at energies 60 keV and 100 keV of a clinical image. From top to bottom: the truths, the results after 100 iterations by the proposed algorithm AFIRE with different selections of $\{C_d^{[q]}\}_{d,q=1}^{D,Q}$.}
    \label{fig:discuss_recon_hatf}
\end{figure}

These results indicate that the convergence, accuracy and efficiency of the proposed algorithm AFIRE  remain robust across different choices of $\{C_d^{[q]}\}_{d,q=1}^{D,Q}$, and naturally the selection of the special point $\hat{\mtbf}$ is also flexible. This enhances the applicability of the algorithm for various practical scenarios without compromising performance.

\subsection{Implementation of the approximated inverse $\mtbP^{+}$}

In \cref{sec:Numerical}, we employed the FBP with 1-bandwidth Ram--Lak filter to implement the approximated inverse $\mtbP^{+}$. Besides that, the proposed algorithm AFIRE also allows us to adopt different implementations for $\mtbP^{+}$. Specifically, we further consider the other two different implementation methods: conjugate gradient (CG) method with 20 inner iterations and limited-memory BFGS (L-BFGS) method with 60 inner iterations. Note that, the numbers of the inner iterations are not fixed, which can be modified on demand. All the implementations are running with Operator Discretization Library (ODL) \cite{adler2018odlgroup}. 

Here, the numerical experiments are carried out to evaluate the feasibility and impact of the different implementations for $\mtbP^{+}$ on the performance of the proposed algorithm. 

We still set $\hat{\mtbf}=\bzero$. After running 100 iterations with \cref{alg:solve_dd_prob}, the obtained basis images and VMIs using those different implementations are shown in \cref{fig:discuss_recon_P_inv}, which are quite close to the corresponding truths. 
\begin{figure}[htbp]
    \centering
    \includegraphics[width=1.\textwidth]{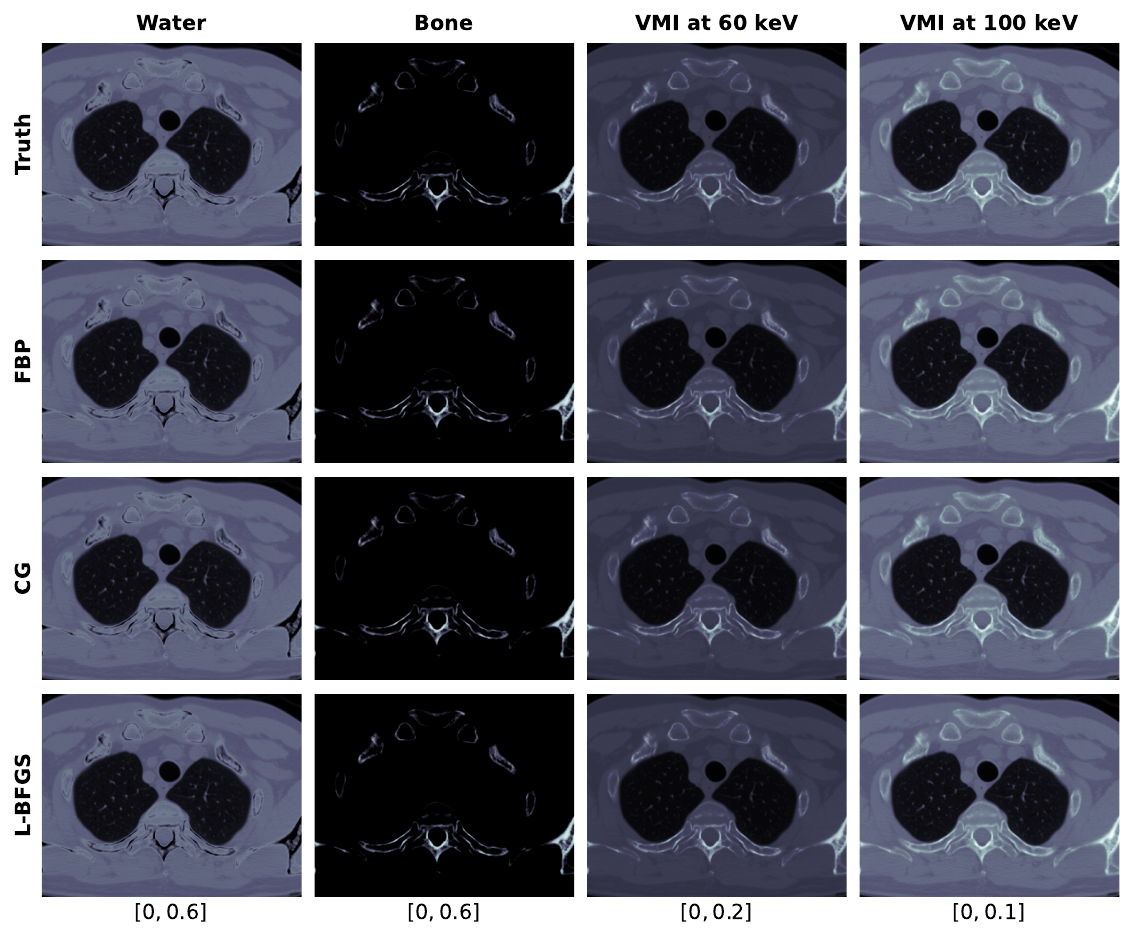}
    \vspace{-5mm}
    \caption{From left to right: basis images of water and bone, VMIs at energies 60 keV and 100 keV of a clinical image. From top to bottom: the truths, the results after 100 iterations by the proposed algorithm AFIRE with the implementations of FBP, CG and L-BFGS for the approximated inverse $\mtbP^{+}$.}
    \label{fig:discuss_recon_P_inv}
\end{figure}
\begin{figure}[htbp]
    \centering
    \includegraphics[width=0.49\textwidth]{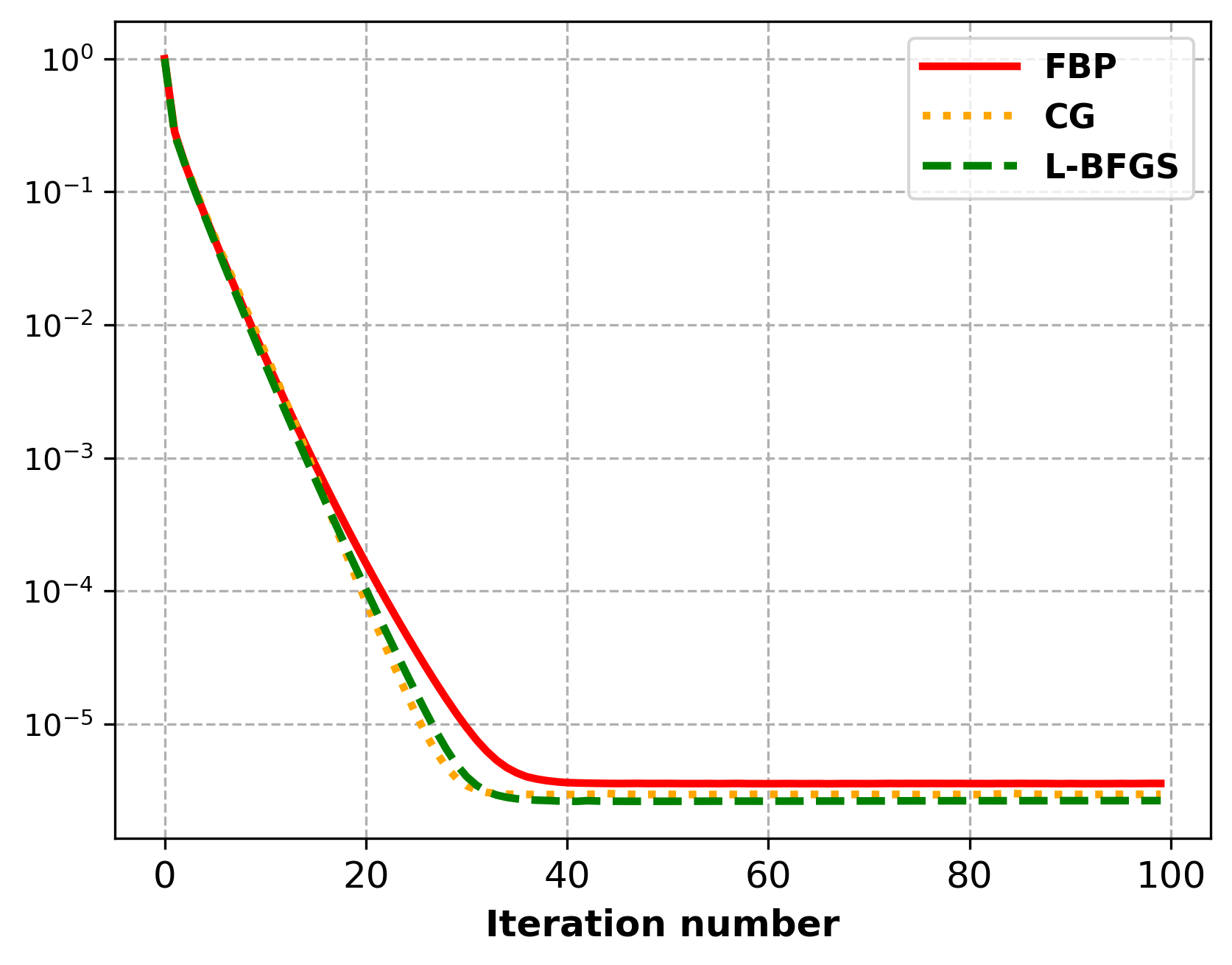}
    \includegraphics[width=0.49\textwidth]{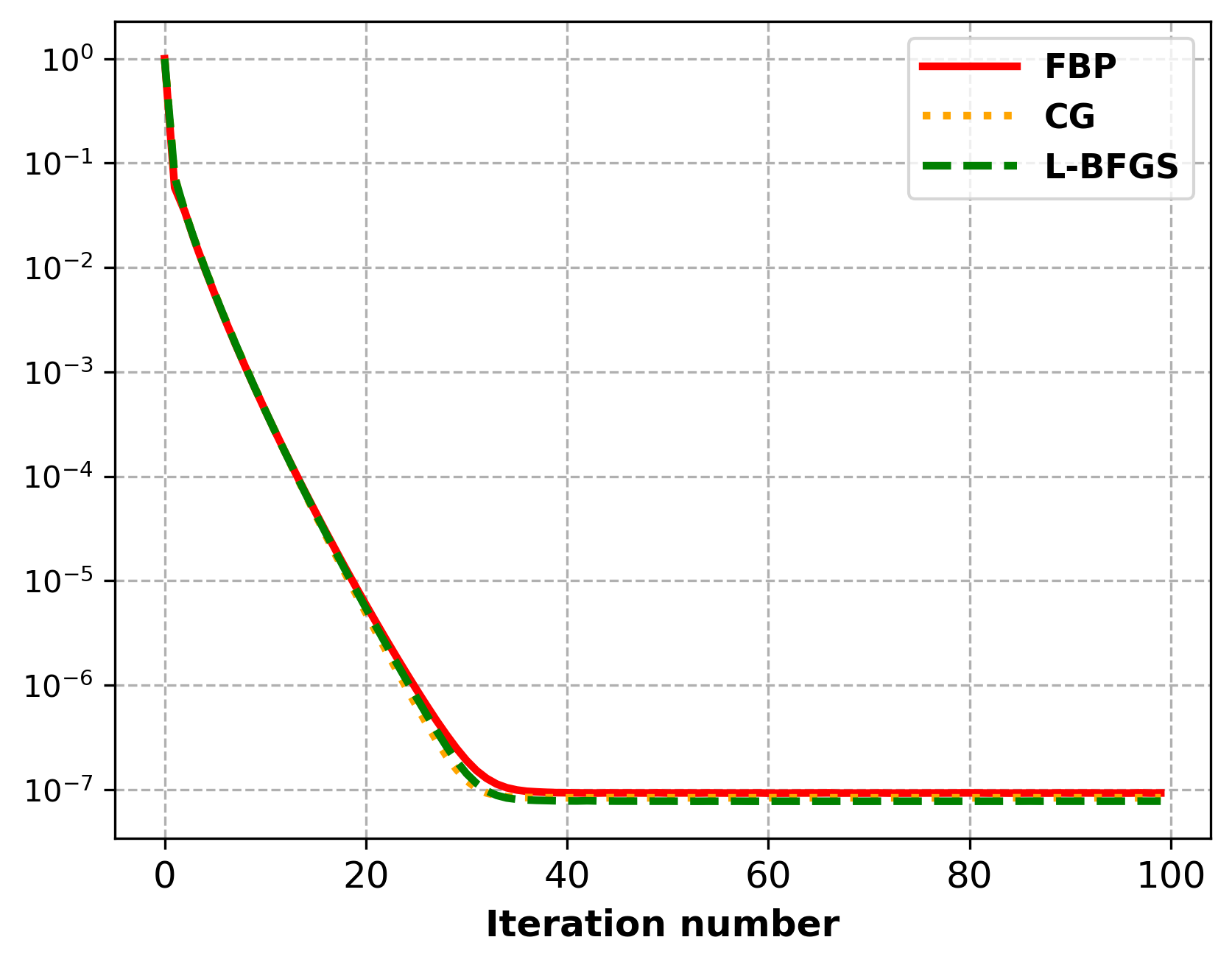} 
    \caption{Metrics {\rm RE$_{\mtbf}^{k}$} (left) and {\rm RE$_{\mtbg}^{k}$} (right) are plotted in semi-log scale as functions of iteration numbers for reconstructing basis images of a clinical image in \cref{fig:discuss_recon_P_inv} by the proposed algorithm with different implementations of FBP, CG and L-BFGS for the approximated inverse $\mtbP^{+}$.}
    \label{fig:discuss_RE_choice_P_inv}
\end{figure}
Moreover, as indicated in \cref{fig:discuss_RE_choice_P_inv}, the metric curves of $\text{RE}_{\mtbf}^{k}$ and $\text{RE}_{\mtbg}^{k}$ reveal the similar convergence behavior among those implementations. The implementation that requires inner iterations, such as CG and L-BFGS, achieves a slightly lower $\text{RE}_{\mtbf}^{k}$ than FBP after the same number of outer iterations. In contrast, as evidenced by the comparison of running time in \cref{tab:time_P_inv}, the FBP implementation is faster in a single (outer) iteration and when the errors $\text{RE}_{\mtbf}^{k}$ and $\text{RE}_{\mtbg}^{k}$ reach $10^{-5}$. 

As described in \cref{rem:trade_off}, a more precise implementation of $\mtbP^{+}$ may improve the convergence rate. However, it may increase the computational cost or reduce the efficiency. Therefore, there is a trade-off between the reconstructed accuracy and computational efficiency.

\begin{table}[htbp]
    \centering
    \caption{The average time (in seconds) of 10 independent runs by the proposed algorithm AFIRE with different implementations for $\mtbP^{+}$ to accomplish various tasks in reconstructing basis images of a clinical image in \cref{fig:discuss_recon_P_inv}.}
    \vspace{5mm}
    \begin{tabular}{cccc}
    \Xhline{1pt}
    Time (s) & FBP & CG & L-BFGS \\
    \Xhline{0.5pt}
    Single (outer) iteration & $0.47$ & 0.80 & 3.22  \\[0.5mm]
    $\text{RE}_{\mtbf}^{k}$ reaches $10^{-5}$ & $13.99$ & 20.70 & 86.81  \\[0.5mm]
    $\text{RE}_{\mtbg}^{k}$ reaches $10^{-5}$ & $8.86$ & 15.13 & 61.09  \\[0.5mm]
    \Xhline{1pt}
    \end{tabular}
    \label{tab:time_P_inv}
\end{table}

\subsection{Extension to the geometric-consistent situation}

From the derivation in \cref{sec:propose_alg}, we know that the proposed algorithm AFIRE can be naturally extended to the geometric-consistent case. To illustrate this, what is different from before is that the identical equidistant fan-beam geometry is used for both 80-kV and 140-kV spectra. There are totally 900 views uniformly distributed over $[0,2\pi)$ and 1086 equal-width detectors positioned on the opposing curved detector bank. 

The proposed algorithm is used to reconstruct the basis images from the noiseless data. In the geometric-consistent case, the two-step DDD method can be used as a benchmark algorithm for comparison. In the first step, to estimate the basis sinograms  in data domain, we use Newton's method to solve a lot of very small-scale independent nonlinear systems along different X-ray paths. When the iteration number is $10$, the relative error has arrived at $10^{-14}$. In the second step, to reconstruct the basis images, the FBP with 1-bandwidth Ram--Lak filter is applied to invert the finally obtained basis sinograms. To show the intermediate results, the FBP method is also utilized to reconstruct the basis images from the estimated basis sinograms at each Newton iteration in the first step.

After running \cref{alg:solve_dd_prob} for 100 iterations with setting $\hat{\mtbf}=\bzero$, the obtained basis images and VMIs are shown in \cref{fig:discuss_recon_geometry_consis}, as well as the corresponding results gained by the two-step DDD method. Furthermore, the relative error curves are presented in \cref{fig:discuss_RE_geometry_consis}. These comparisons demonstrate that the performance of the proposed algorithm far exceeds the benchmark two-step DDD method. Hence, the proposed algorithm cannot only be successfully extended to deal with the geometric-consistent case of MSCT image reconstruction, but also achieve very high accuracy and efficiency.
\begin{figure}[htbp]
    \centering
    \includegraphics[width=1.\textwidth]{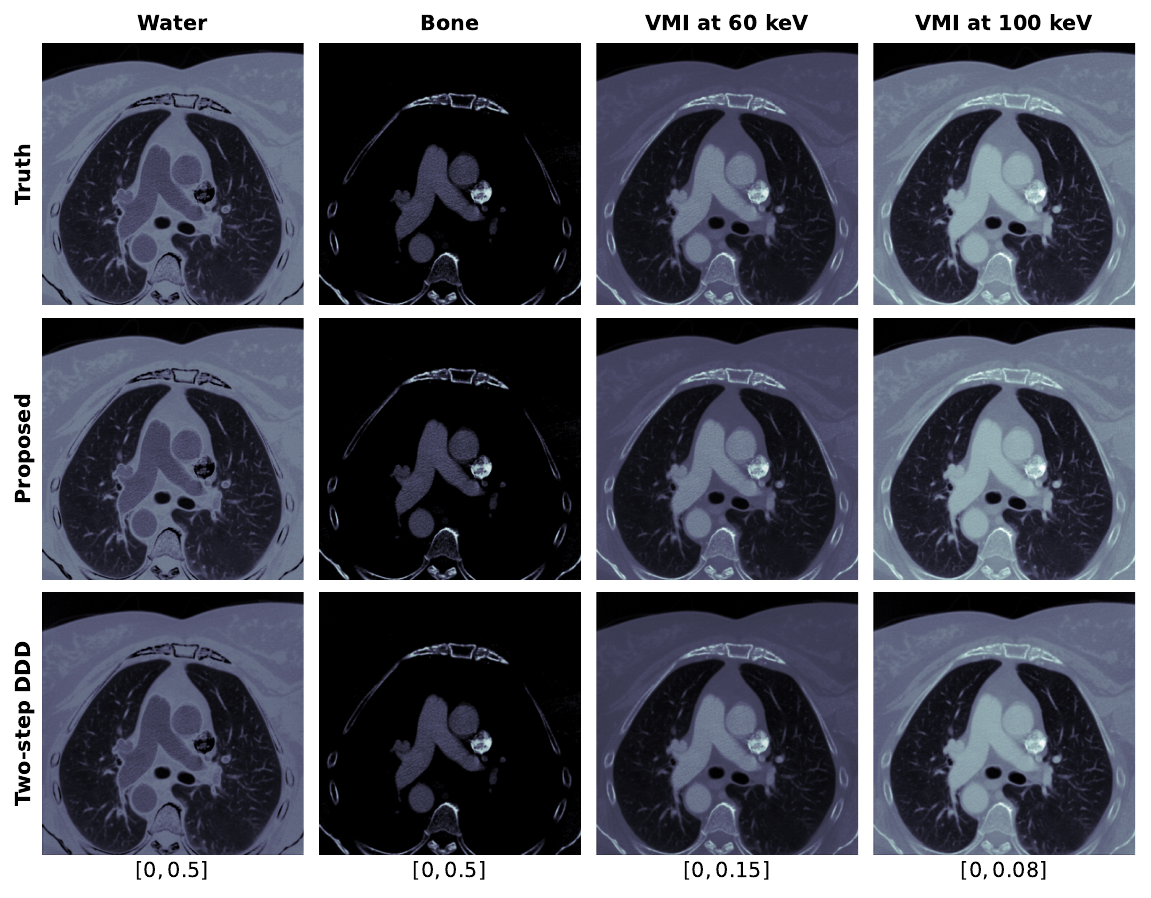}
    \vspace{-5mm}
    \caption{From left to right: basis images of water and bone, VMIs at energies 60 keV and 100 keV of a clinical image. From top to bottom: the truths, the results by the proposed algorithm AFIRE after 100 iterations and the benchmark two-step DDD method.}
    \label{fig:discuss_recon_geometry_consis}
\end{figure}

\begin{figure}[htbp]
    \centering
    \includegraphics[width=0.49\textwidth]{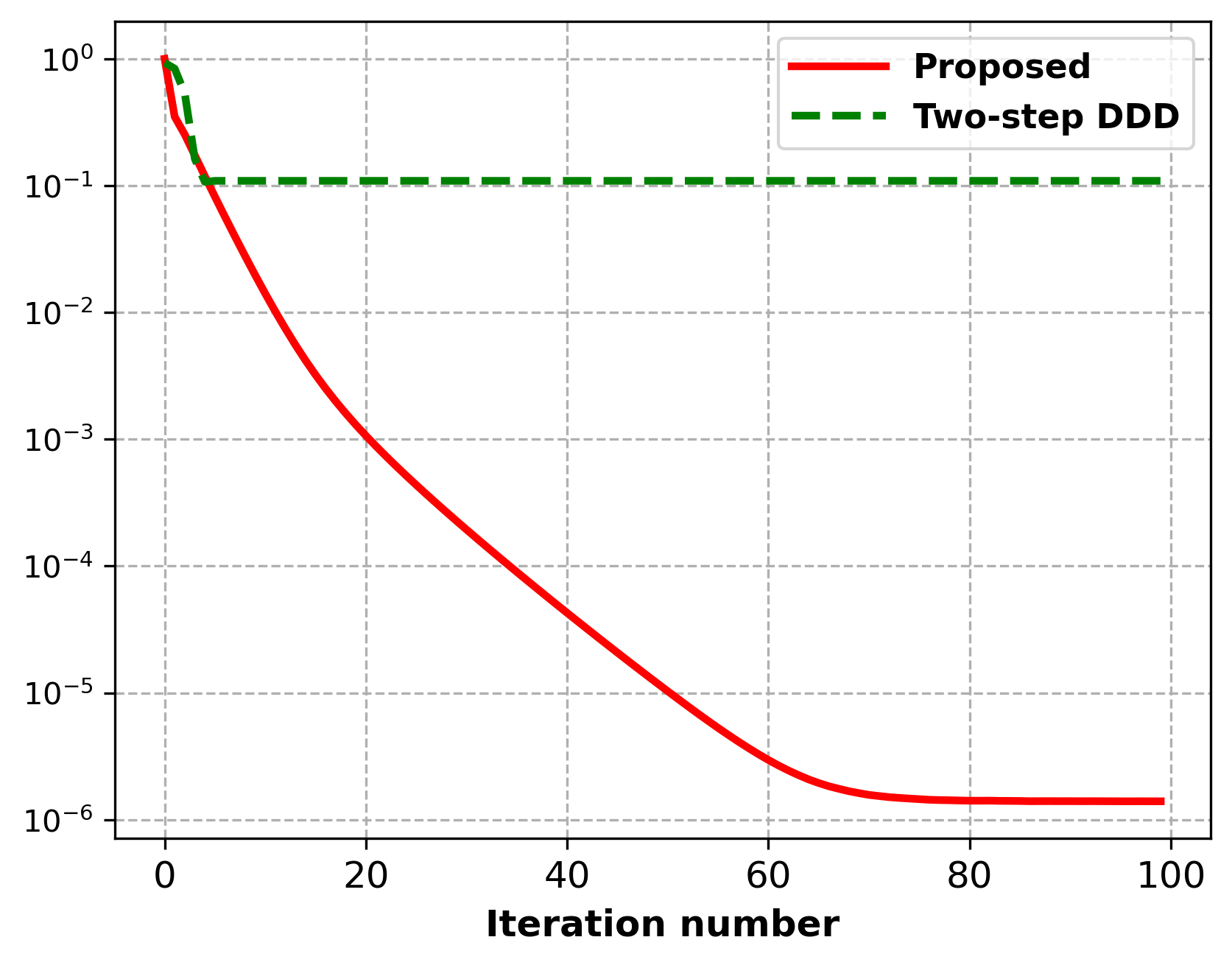}
    \includegraphics[width=0.49\textwidth]{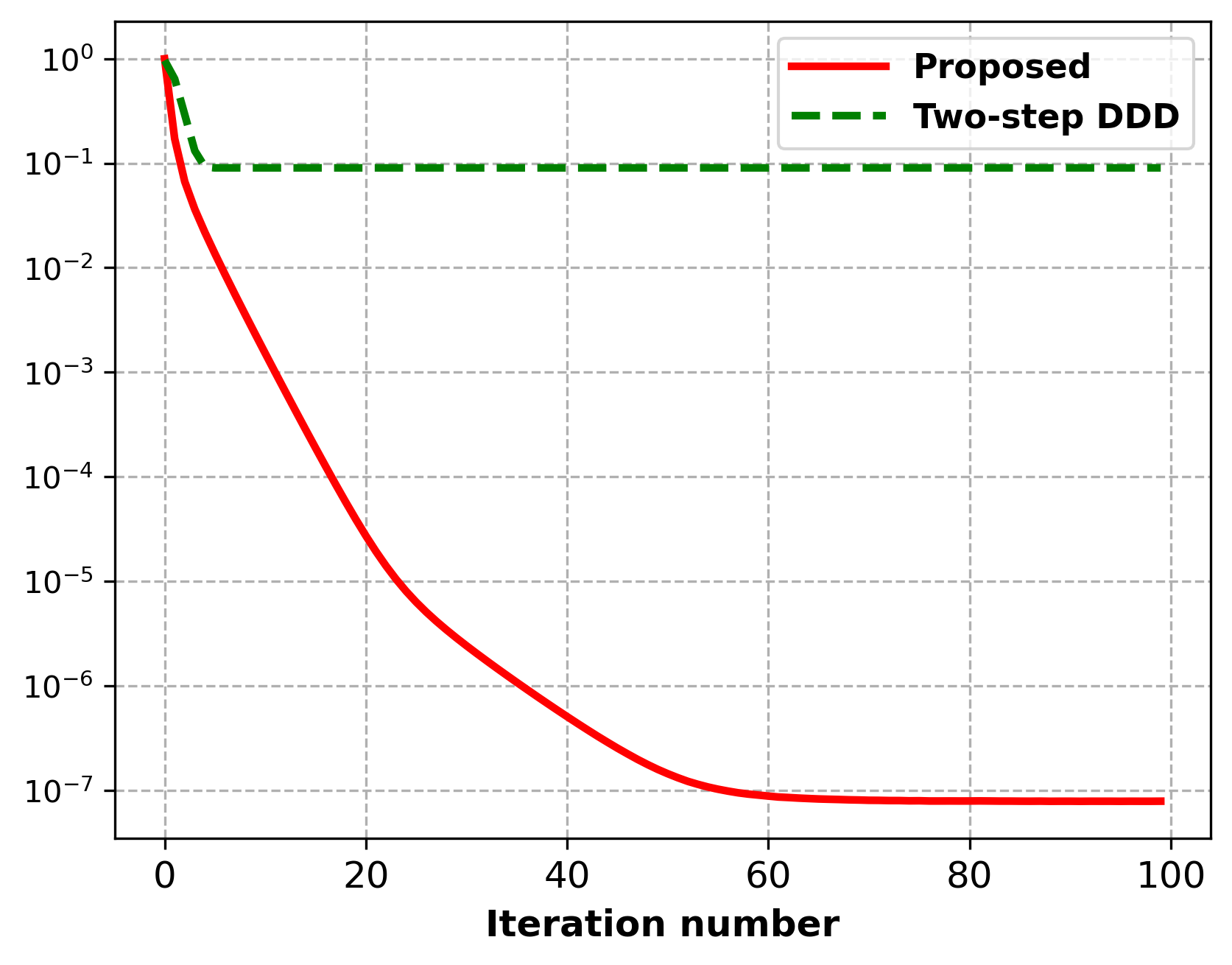} 
    \caption{Metrics {\rm RE$_{\mtbf}^{k}$} (left) and {\rm RE$_{\mtbg}^{k}$} (right) are plotted in semi-log scale as functions of iteration numbers for the geometric-consistent MSCT reconstruction problem using the proposed algorithm AFIRE and the benchmark two-step DDD method.}
    \label{fig:discuss_RE_geometry_consis}
\end{figure}

\subsection{Validation with the basis images generated by randomly sampled Gaussian distribution}

In this work, the AFIRE algorithm is developed for the geometric-inconsistent multispectral CT with a mildly full scan. In this case, the number of measurements is comparable to the total size of the employed basis images. The sufficient projection data that obtained in this case allows us to treat the basis images as entirely unknown variables, and thus can achieve an efficient reconstruction without priors on the basis images. 

To further validate the statement above, we conduct additional experiments using the basis images generated by randomly sampled Gaussian distribution as truths for the water and bone basis images. Specifically, we generate 1000 random array pairs as the basis images using \texttt{numpy.random} with different seeds. All the other experimental settings are consistent with those described in \cref{subsec:noiseless_forbild}. After running the AFIRE algorithm for 100 iterations, we plot the scattered distributions of the metrics of relative errors $\text{RE}_{\mtbf}^{k}$ and $\text{RE}_{\mtbg}^{k}$ compared to those different random truths in \cref{fig:REk_rdm_truth}. It demonstrates that the AFIRE algorithm can indeed arrive at a high-precision reconstruction even for the randomly generated images. This highlights the robustness and effectiveness of the proposed AFIRE algorithm. 
\begin{figure}[htbp]
    \centering
    \includegraphics[width=0.49\textwidth]{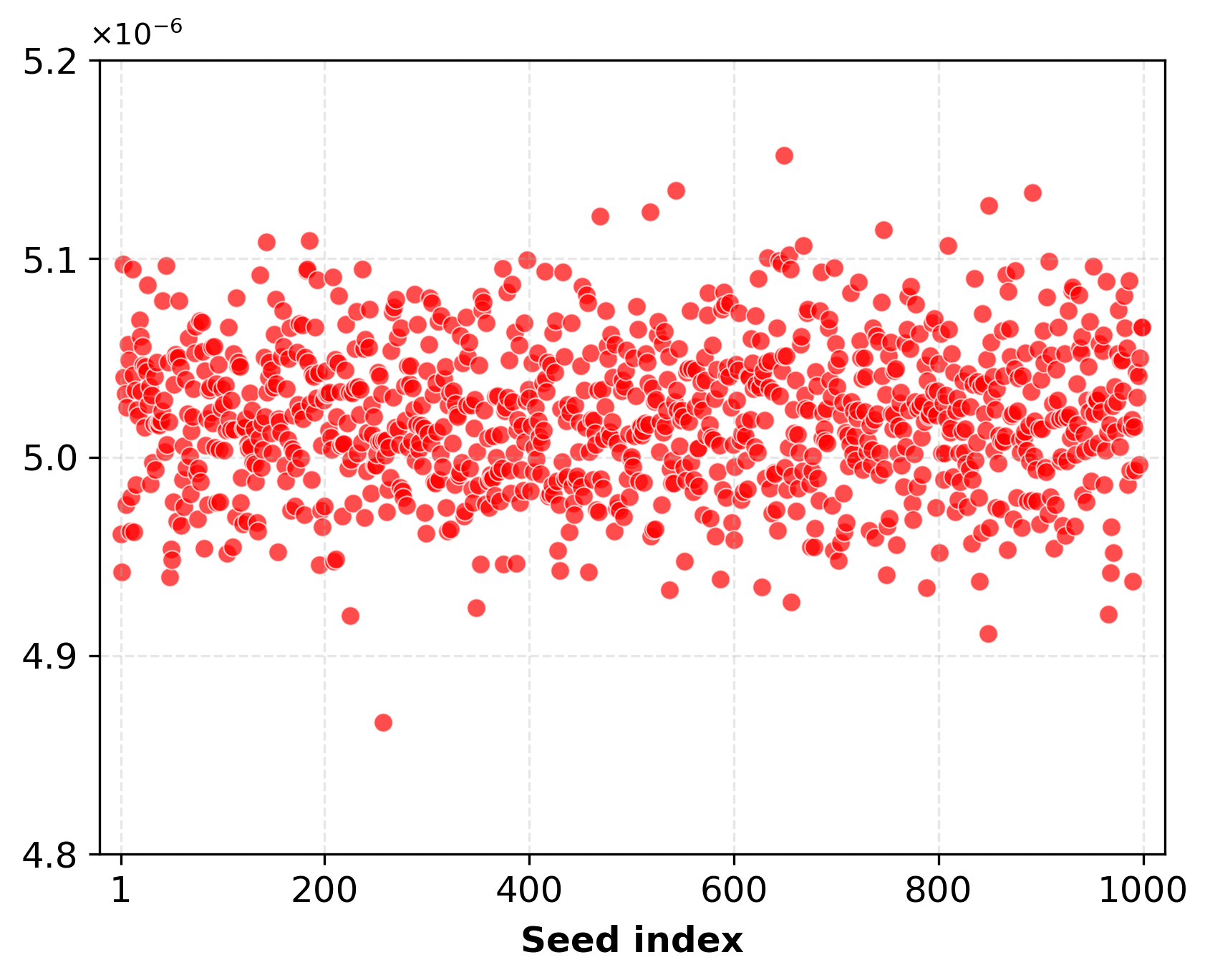}
    \includegraphics[width=0.49\textwidth]{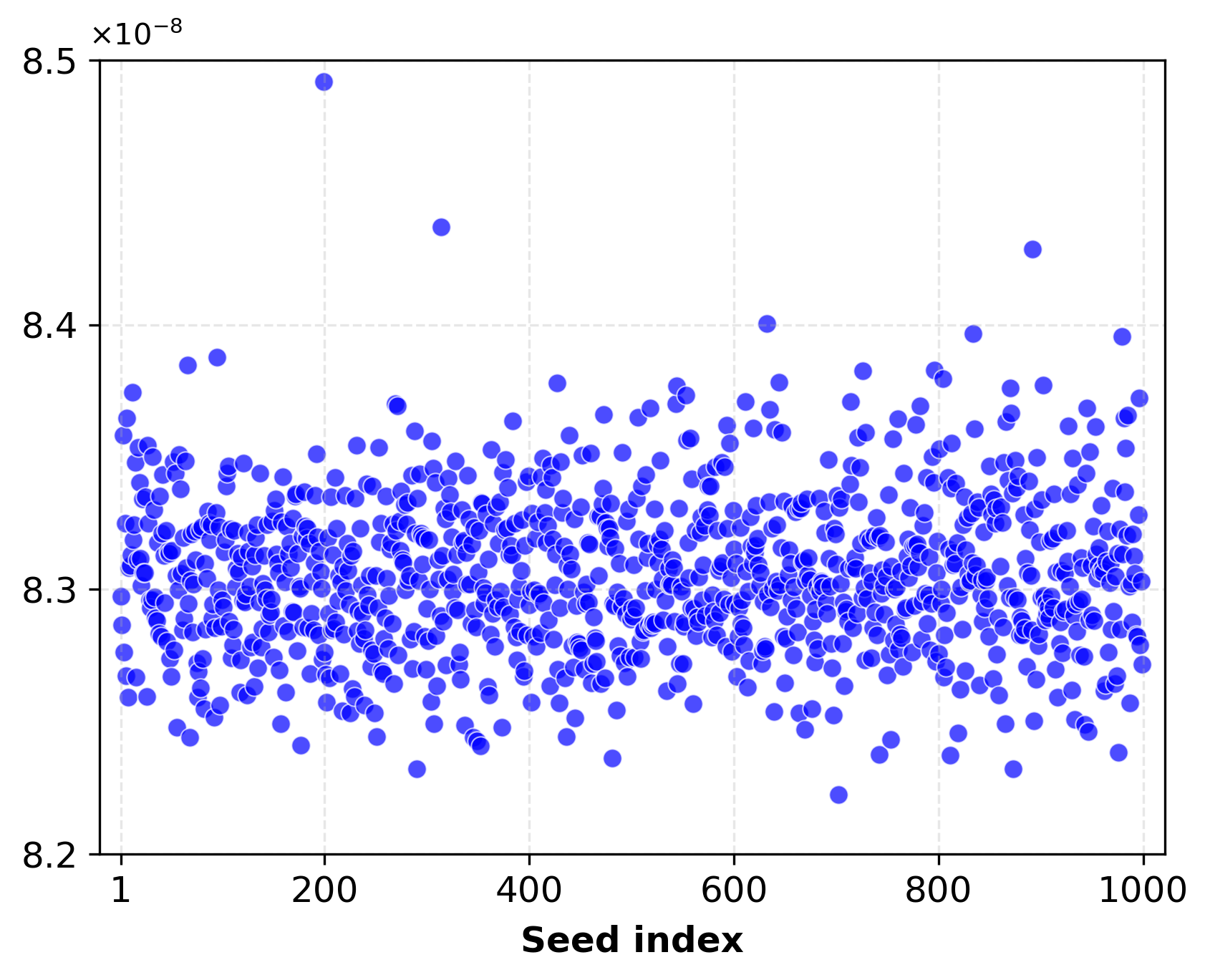}
    \caption{The scattered diagrams of the metrics of relative errors {\rm RE$_{\mtbf}^{k}$} (left) and {\rm RE$_{\mtbg}^{k}$} (right) for reconstructing the basis images generated by randomly sampled Gaussian distribution by using the proposed AFIRE algorithm after 100 iterations from the noiseless data.}
    \label{fig:REk_rdm_truth}
\end{figure}

\section{Conclusion}\label{sec:conclucsion}

In this work, the image reconstruction algorithm was investigated for the nonlinear MSCT when the scanning geometries under different X-ray energy spectra are inconsistent or mismatched. 

We studied the associated CC (resp. DD)-data model, and discovered that the derivative operator  (gradient) of the involved nonlinear mapping at some special points, for example, at zero, can be represented as a composition (block multiplication) of a diagonal operator (matrix) composed of X-ray transforms (projection matrices) and a very small-scale matrix. Based on the insights, we proposed an accurate and fast image reconstruction algorithm called AFIRE respectively from the continuous, discrete and actual-use perspectives by leveraging the simplified Newton method. Under proper conditions, we further proved the convergence of the proposed algorithm AFIRE. 

Furthermore, several numerical experiments were carried out to verify that the proposed algorithm can accurately and effectively reconstruct the basis images in geometric-inconsistent DECT for noiseless and noisy projection data in the case of mildly full scan. Through numerical comparisons, it was shown that the proposed algorithm significantly outperforms some existing methods such as NCPD, NKM, IFBP and INTRPL with respect to the accuracy and efficiency. 

We also discussed the flexibility and extensibility of the proposed algorithm. It was demonstrated that the proposed algorithm is robust and flexible with regard to the selection of the special point $\hat{\mtbf}$ and the implementation of the approximated inverse $\mtbP^{+}$, and also can indeed arrive at a high-precision reconstruction even for the randomly generated images. Moreover, for the geometric-consistent situation, it was shown that the proposed algorithm has significant advantages over the benchmark two-step DDD method.

To the best of our knowledge, when the scanning geometries under different X-ray energy spectra are inconsistent or mismatched in MSCT, this is the first literature that uses the characteristics of the derivative operator (gradient) of the associated nonlinear mapping to construct the algorithm for accurate and fast image reconstruction. Moreover, this methodology may also provide insights into designing the algorithm of the same type for the corresponding low-dose or sparse-view scenarios.

\appendix
\section{Appendix}

\subsection{The proof of \cref{prop:Frechet_derivative}}\label{proof:G_F_diff}
\begin{proof}
 For any $(\bdomega,\bx)\in\mathbb{T}^n_q$, $K^{[q]}(\bdf)(\bdomega,\bx)$ can be rewritten as the function $H_q:\Real^D\rightarrow\Real$ which is given by
    \begin{align*}
       \bdz &\longmapsto \ln \int_{0}^{E_{\max}} s^{[q]}(E)\exp\left(-\sum_{d=1}^D  b_d(E) z_d \right)\rmd E,
    \end{align*}
    where $\bdz=[z_1,\ldots,z_D]^{\tra}$ and $z_d=\Xray_q f_d(\bdomega,\bx)$. It is easy to compute that 
    \begin{equation}\label{eq:partial_Hq}
        \frac{\partial H_q(\bdz)}{\partial z_d} = -\int_{0}^{E_{\max}} \frac{s^{[q]}(E)\exp\left( -\sum_{d=1}^D  b_d(E) z_d\right)}{ \int_{0}^{E_{\max}} s^{[q]}(E)\exp\left(-\sum_{d=1}^D  b_d(E) z_d \right) \rmd E} b_d(E)~\rmd E.
    \end{equation} 
    By \cref{eq:partial_Hq} and the nonnegative bounded property of functions $\{b_d\}_{d=1}^D$, we have 
    \begin{equation}\label{eq:nabla_Hq}
        \|\nabla H_q(\bdz)\|_1 =\sum_{d=1}^D \left|\frac{\partial H_q(\bdz)}{\partial z_d}\right| \le \sum_{d=1}^D \sup_{E} b_d(E).
    \end{equation} 
    Notice that $K^{[q]}(\bzero)(\bdomega,\bx)=0$ for any $(\bdomega,\bx)\in\mathbb{T}^n_q$, and using  \cref{eq:nabla_Hq}, we obtain  
    \begin{multline}\label{eq:K_qf}
        \big{|} K^{[q]}(\bdf)(\bdomega,\bx) \big{|}  = \big{|} K^{[q]}(\bdf)(\bdomega,\bx)-K^{[q]}(\bzero)(\bdomega,\bx)\big{|}    =\left|H_q(\bdz) - H_q(0) \right|\\ \le \left(\sup_{\bdz}\|\nabla H_q(\bdz)\|_1\right)\|\bdz\|_1 \le \left(\sum_{d=1}^D \sup_{E} b_d(E)\right) \sum_{d=1}^D\left|\Xray_q f_d(\bdomega,\bx)\right|. 
    \end{multline}
    Consequently, for any $\bdf\in\mathbfcal{F}^D$, by \cref{eq:K_qf}, we have 
    \begin{align*}
        \|K^{[q]}(\bdf)\|_{\mathcal{G}_q} &= \left(\int_{\mathbb{S}^{n-1}_q}\int_{\Omega_q} \big{|}  K^{[q]}(\bdf)(\bdomega,\bx) \big{|}^2  \rmd \boldsymbol{\mu}_{\bx} \rmd \bdomega\right)^{1/2} \\
        &\le  \left(\sum_{d=1}^D \sup_{E} b_d(E)\right) \left( \int_{\mathbb{S}^{n-1}_q}\int_{\Omega_q} \Bigl(\sum_{d=1}^D \left|\Xray_q f_d(\bdomega,\bx)\right|\Bigr)^2 \rmd \boldsymbol{\mu}_{\bx} \rmd \bdomega \right)^{1/2} \\
        &\le  \sqrt{D}\left(\sum_{d=1}^D \sup_{E} b_d(E)\right) \left( \int_{\mathbb{S}^{n-1}_q}\int_{\Omega_q} \sum_{d=1}^D \left|\Xray_q f_d(\bdomega,\bx)\right|^2 \rmd \boldsymbol{\mu}_{\bx} \rmd \bdomega \right)^{1/2} \\
        &\le \sqrt{D}\left(\sum_{d=1}^D \sup_{E} b_d(E)\right) \sum_{d=1}^D \|\Xray_q f_d\|_{\mathcal{G}_q}, 
    \end{align*}
    where $\rmd \boldsymbol{\mu}_{\bx}$ denotes the Euclidean measure on the hyperplane $\bdomega^{\perp}$. Since $\Xray_q f_d\in\mathcal{G}_q$ for each $d$, we obtain $K^{[q]}(\bdf)\in\mathcal{G}_q$. 

    For any $\bdh=[h_1,\ldots,h_D]  \in \mathbfcal{F}^D$, by the chain rule, we compute 
    \begin{equation}\label{eq:G_derivative_K}
    \partial_{\varepsilon} K^{[q]} (\bdf + \varepsilon \bdh)\mid_{\varepsilon=0} = - \int_0^{E_{\max}}  \Phi^{[q]}(\bdf; E) \left( \sum_{d=1}^D b_d(E) \Xray_q h_d\right) \rmd E,
    \end{equation}
    where $\Phi^{[q]}(\bdf; E)$ is given as in \cref{eq:Phi_f} such that  
    \[\int_{0}^{E_{\max}} \Phi^{[q]}(\bdf; E)~ \rmd E = 1 \quad \text{for}~ q=1,\ldots,Q.
    \]
    Then, from \cref{eq:G_derivative_K}, we have the following estimate 
    \begin{align}\label{eq:G_derivative_K_estimate}
    \big\|\partial_{\varepsilon} K^{[q]} (\bdf + \varepsilon \bdh)\mid_{\varepsilon=0}\big\|_{\mathcal{G}_q}&=\big\|\sum_{d=1}^D \int_0^{E_{\max}}  \Phi^{[q]}(\bdf; E) b_d(E) \Xray_q h_d ~\rmd E\big\|_{\mathcal{G}_q}\nonumber\\ 
    &\le \left(\sup_{d,E} b_d(E)\right)\sum_{d=1}^D \big\| \int_0^{E_{\max}}  \Phi^{[q]}(\bdf; E) \Xray_q h_d ~\rmd E\big\|_{\mathcal{G}_q} \nonumber\\ 
    &= \left(\sup_{d,E} b_d(E)\right) \sum_{d=1}^D \left\|\Xray_q h_d\right\|_{\mathcal{G}_q}\nonumber\\
    &\le \left( \sup_{d,E} b_d(E)\right) C  \sum_{d=1}^D \|h_d\|_{\mathcal{F}}. 
    \end{align}
Note that the last inequality follows readily from the result in \cite[theorem 2.10]{natterer2001mathematical}, and the $C$ is a positive constant depending only on $\alpha = -1/2$ and the domain $\domain$. Hence, using \cref{eq:G_derivative_K} and \cref{eq:G_derivative_K_estimate}, we conclude that $K^{[q]}$ is G\^ateaux differentiable, and its G\^ateaux derivative is computed as \cref{eq:Frech_deriva}. 
    
Furthermore, from the continuity of $\Xray_q$, it is easy to obtain that the mapping $\bdf \mapsto \Diff K^{[q]} (\bdf,\cdot)$ is continuous. By \cite[proposition 3.2.15]{Nonlinear_ana07}, $K^{[q]}$ is Fréchet differentiable and its Fréchet derivative is equal to $\Diff K^{[q]}(\bdf,\cdot)$. By the definition of the norm of Cartesian space, we can further prove that $\bdK$ is also Fréchet differentiable. 
\end{proof}

\subsection{The proof of \cref{lem:continuous}}\label{proof:cc_lemma}
\begin{proof}    
For any $\bdf$ and $(\bdomega,\bx)$, we can define the matrix 
\[
\Psi(\bdf)(\bdomega,\bx) := \bigl(\psi^{[q]}_d(\bdf)(\bdomega,\bx)\bigr)\in\Real^{Q\times D},
\] 
where
\begin{equation*}
    \psi^{[q]}_d(\bdf)(\bdomega,\bx) = \int_{0}^{E_{\max}} \Phi^{[q]}(\bdf; E)(\bdomega,\bx) b_d(E)~\rmd E.
\end{equation*}
Since $\Phi^{[q]}(\bdf; \,\cdot\,)$ as given in \cref{eq:Phi_f} is a nonnegative normalized function with respect to $E$, together with the given conditions, we have
\begin{align*}
    \|\Psi(\bdf)(\bdomega,\bx) - &\bvarphi\|_F^2 \\
    &=\sum_{q=1}^Q \sum_{d=1}^D \left(\int_0^{E_{\max}}  \big(\Phi^{[q]}(\bdf; E)(\bdomega,\bx) - \Phi^{[q]}(\hat{\bdf}; E)\big)  b_d(E) ~\rmd E\right)^2\\
    &\le \eta^2 \sum_{q=1}^Q \sum_{d=1}^D \left(\int_0^{E_{\max}}  \Phi^{[q]}(\hat{\bdf}; E) b_d(E) ~\rmd E\right)^2 = \eta^2 \|\bvarphi\|_F^2.
\end{align*}
Subsequently, for any $\bdz\in\Real^D$, by Cauchy--Schwarz inequality and the property of the matrix norm, we have
\begin{align*}
\|(\Psi(\bdf)(\bdomega,\bx)- \bvarphi)\bdz\|_2 
&\le \|\Psi(\bdf)(\bdomega,\bx)- \bvarphi\|_2  \|\bdz\|_2 
\le \|\Psi(\bdf)(\bdomega,\bx)- \bvarphi\|_F \|\bdz\|_2 \\
& \le \eta\|\bvarphi\|_F \|\bdz\|_2 \le \eta\kappa_F(\bvarphi) \|\bvarphi\bdz\|_2.
\end{align*}
Hence, for any $\bdf,\bdh\in\mathbfcal{F}^D$, based on the formula of Fréchet derivative of $\bdK$ and the above estimation, we obtain
\begin{align*}
\big\| \Diff \bdK(\bdf,\bdh) - \Diff \bdK(\hat{\bdf},&\bdh)\big\|_{\mathbfcal{G}^Q}^2 \\ 
&= \int_{\mathbb{S}^{n-1}} \int_{\bdomega^{\perp}} \left\|\left(\Psi(\bdf)(\bdomega,\bx)- \bvarphi\right)\bXray\bdh(\bdomega,\bx)\right\|_2^2 \rmd \boldsymbol{\mu}_{\bx} \rmd \bdomega \\
&\le \int_{\mathbb{S}^{n-1}} \int_{\bdomega^{\perp}} (\eta\kappa_F(\bvarphi))^2 \|\bvarphi\bXray\bdh(\bdomega,\bx)\|_2^2 \rmd \boldsymbol{\mu}_{\bx} \rmd \bdomega \\
&= (\eta\kappa_F(\bvarphi))^2 \|\Diff \bdK(\hat{\bdf},\bdh)\|_{\mathbfcal{G}^Q}^2,
\end{align*}
where $\bXray\bdh(\bdomega,\bx)=[\Xray_1 h_1(\bdomega,\bx),\ldots,\Xray_D h_D(\bdomega,\bx)]^{\tra}\in\Real^D$. The desired result is obtained. 
\end{proof}

\subsection{The proof of \cref{lem:discrete}}\label{proof:dd_lemma}
\begin{proof}
    Notice that for any $\mtbf$, $\mtbw_j^{[q]}(\mtbf) \ge\bzero$ and $\bigl({\mtbw_j^{[q]}(\mtbf)}\bigr)^{\tra} \boldsymbol{1}=1$, which means that $\mtbw_j^{[q]}(\mtbf) \in \Real^M_+ $ is in the unit simplex. Hence, by the gradient calculation in \cref{subsubsec:dd_alg}, and using the given conditions, we obtain
    \begin{align*}
        \big\|\nabla \mtbK(\hat{\mtbf}) - \nabla \mtbK(\mtbf)\big\|_F^2 =& \sum_{q=1}^Q \sum_{j=1}^{J_q} \|\nabla K_j^{[q]}(\hat{\mtbf}) - \nabla K_j^{[q]}(\mtbf)\|_2^2\\ 
        =& \sum_{q=1}^Q \sum_{j=1}^{J_q} \sum_{d=1}^D |\big(\mtbw_j^{[q]}(\hat{\mtbf}) - \mtbw_j^{[q]}(\mtbf)\big)^{\tra} \mtbb_d \mtbp^{[q]}_j |^2 \\ 
        =& \sum_{q=1}^Q \sum_{j=1}^{J_q} \sum_{d=1}^D |\mtbb_d^{\tra} {\hat{\mtbw}}^{[q]} - \mtbb_d^{\tra}{\mtbw_j^{[q]}(\mtbf)} |^2 \|\mtbp^{[q]}_j\|_2^2 \\ 
        \le&\sum_{q=1}^Q \sum_{j=1}^{J_q} \sum_{d=1}^D \tilde{\eta}^2 |\mtbb_d^{\tra} \hat{\mtbw}^{[q]} |^2 \|\mtbp^{[q]}_j\|_2^2 \\ 
        =& \tilde{\eta}^2\|\nabla \mtbK(\hat{\mtbf})\|_F^2.
    \end{align*}
    Consequently, by the property of the matrix norm, the above estimate in order, we have 
    \begin{align*}
        \big\|\nabla \mtbK(\hat{\mtbf})\bdh - \nabla \mtbK(\mtbf)\mtbh\big\|_2 &\le \|\nabla \mtbK(\hat{\mtbf}) - \nabla \mtbK(\mtbf)\|_F \|\mtbh\|_2 
        \le \tilde{\eta} \|\nabla \mtbK(\hat{\mtbf})\|_F \|\mtbh\|_2 \\ 
        &\le \tilde{\eta} \|\nabla \mtbK(\hat{\mtbf})\|_F \sigma_{\min}^{-1}\bigl(\nabla \mtbK(\hat{\mtbf})\bigr) \|\nabla \mtbK(\hat{\mtbf})\mtbh\|_2 \\ 
        &= \tilde{\eta} \kappa_F(\mtbP\otimes \bphi) \|\nabla \mtbK(\hat{\mtbf})\mtbh\|_2,
    \end{align*}
which completes the proof.
\end{proof}

\bibliographystyle{plain}
\bibliography{MCTreferences}

\end{document}